\documentclass[reqno,10pt, centertags,draft]{amsart}
\usepackage{amsmath,amsthm,amscd,amssymb,latexsym,upref}
\usepackage{latexsym}
\usepackage{lscape}
\usepackage{url}

\newcommand\numberthis{\addtocounter{equation}{1}\tag{\theequation}}

\newcommand{\beq}{\begin{equation}}
\newcommand{\enq}{\end{equation}}

\makeatletter
\def\theequation{\@arabic\c@equation}

\newcommand{\bbR}{{\mathbb{R}}}

\newcommand{\lb}{\label}

\newcommand{\bu}{{\mathbf u}}
\newcommand{\bv}{{\mathbf v}}

\newcommand{\bn}{{\mathbf n}}

\newcommand{\bx}{{\mathbf x}}
\newcommand{\bt}{{\mathbf t}}

\newcommand{\bi}{\bibitem}

\numberwithin{equation}{section}

\renewcommand{\div}{\operatorname{div}}

\newcommand{\nt}{{\nabla^\perp}}

\newcommand{\curl}{\operatorname{curl}}

\theoremstyle{plain}
\newtheorem{theorem}{Theorem}[section]

\newtheorem{assumption}[theorem]{Assumption}
\newtheorem{lemma}[theorem]{Lemma}
\newtheorem{corollary}[theorem]{Corollary}
\newtheorem{proposition}[theorem]{Proposition}

\theoremstyle{definition}

\newtheorem{example}[theorem]{Example}

\newtheorem{remark}[theorem]{Remark}

\begin{document}
\allowdisplaybreaks
\title[Stability criteria 2D $\alpha$-Euler]{Stability criteria for the 2D $\alpha$-Euler equations}

\thanks{Partially supported by  NSF grant DMS-171098, Research Council of the University of Missouri and the Simons Foundation. We would like to thank Samuel Walsh for many helpful discussions.}
\author[Y. Latushkin]{Yuri Latushkin}
\address{Department of Mathematics,
University of Missouri, Columbia, MO 65211, USA}
\email{latushkiny@missouri.edu}
\urladdr{http://www.math.missouri.edu/personnel/faculty/latushkiny.html}
\author[S. Vasudevan]{Shibi Vasudevan}
\address{International Centre for Theoretical Sciences,
Tata Institute of Fundamental Research, Bengaluru, 560089, India}
\email{shibi.vasudevan@icts.res.in}
\date{\today}
\begin{abstract}
We derive analogues of the classical Rayleigh, Fjortoft and Arnold stability and instability theorems in the context of the 2D $\alpha$-Euler equations.
\end{abstract}
\maketitle
\normalsize
\section{Introduction} \lb{s1}
The $\alpha$-Euler equations (along with its variants such as $\alpha$-Navier Stokes, Camassa-Holm equations) were introduced and studied in a series of foundational papers by C.\ Foias, D.\ Holm, J.\ Marsden, T.\ Ratiu, E.\ Titi and others, see for example, \cite{FHT01}, \cite{HMR98}, \cite{HMR98a} and references therein. It has been studied intensively ever since by many authors from a wide variety of research themes looking at issues such as possible route to proving well posedness for the Euler equations (see \cite{MS01, OS01}), investigating the geometric structure of these equations along the lines introduced by V.\ I.\ Arnold for the Euler equations (see \cite{MRS00, S98}), in turbulence modeling (see, for instance, \cite{CFHOTW98, CS04}) and data assimilation (see, for example, \cite{ANT16} and references therein), amongst other topics. In spite of this, surprisingly little is known about stability of steady state solutions to the $\alpha$-Euler equations (see, however, \cite{PS00}), even though stability theory of the classical 2D Euler equations has been extensively studied, see, e.g., \cite{DZ04, Fr, FS01, FSV97, FY99, GW96, SL05, S00, WG98}. 

Our objectives in this paper are to record basic stability results such as the Rayleigh, Fjortoft and Arnold stability theorems for the $\alpha$-Euler equations. The paper is organized as follows. In Section 2 we formulate the $\alpha$-Euler equations and give preliminaries. In Section 3 we give $\alpha$-analogues of the classical Rayleigh and Fjortoft criteria for 2D parallel shear flows in the channel. Arnold's stability theorems are discussed in Section 4. First, we prove the $\alpha$-version of the first and second instability theorems for the case of multi-connected domains. Our formulation of the second stability theorem directly involves the minimal eigenvalue of the respective $\alpha$-version of the Laplacian, and is based on the use of the Rayleigh-Ritz formula. Next, we discuss stability theorems on the 2-torus and on periodic channels and then provide some examples.
\section{Governing equations and basic setup}
The $\alpha$-Euler equations are a regularization of the classical Euler equations. On a domain $D \subset \mathbb R^{2}$, they are given as follows:
\begin{equation}\lb{EE_alpha}
\begin{split}
\bu_t+\big(\bv\cdot\nabla\big)\bu+ \big(\nabla \bv \big)^{\top} \bu + \nabla \pi=0,\\
(1-\alpha^2 \Delta)\bv=\bu , \\
  \div\bu=\div\bv=0,
\end{split}
\end{equation}
where $\bv:D \to \bbR^{2}$ is the so called filtered velocity and $\bu:D \to \bbR^{2}$ is the actual fluid velocity. Here $\alpha > 0$ is a positive number related to the filter of the flow and $\big(\nabla \bv \big)^{\top}$ represents the transpose of the Jacobi matrix of partial derivatives $\partial \bv_i / \partial x_j$. The pressure $\pi:D \to \bbR$ is related to the actual fluid pressure $p:D \to \bbR$ as \[ \pi= p- \frac{1}{2}|\bv|^2-\frac{\alpha^2} {2}|\nabla \bv|^2.\]
On domains $D$ with a boundary $\partial D$, these equations are supplemented by the boundary conditions, see, for example \cite{HMR98} (formula 8.27, page 65),
\begin{equation}\label{eeabc}
\bv \cdot \bn = 0, \quad (\bn \cdot \nabla \bv) \parallel \bn,
\end{equation}
where $\bn$ is the unit outer normal vector on $\partial D$.
It is also possible to impose $\bv=0$ as the boundary condition. We choose to not work with this as it precludes the possibility of working with steady states which are shear flows with non zero velocity on the boundary. 
Formally, putting $\alpha=0$ gives us the Euler equations
\begin{equation}\lb{EE}
\bu_t+\big(\bu\cdot\nabla\big)\bu+\nabla p=0,\quad \div\bu=0,
\end{equation}
 with the boundary condition 
\begin{equation}\label{eebc}
\bu\cdot\bn=0,
\end{equation}
on $\partial D$.

Since $\div \bv=0$, there exists a stream function $\phi$ such that $\bv=-\nt\phi$, where $\nt=(-\partial_y, \partial_x)$.  We introduce the vorticity 
\begin{equation}\label{vortalpha}
\omega=\curl (1-\alpha^{2}\Delta)\bv.
\end{equation}
Applying $\curl$ to \eqref{EE_alpha}, one obtains the $\alpha$-Euler equation in vorticity form,
\begin{equation}\lb{EEV_alpha}
\omega_t+\bv\cdot\nabla\omega=0.
\end{equation}
Using the fact that $\curl (-\nt)\phi=-\Delta\phi$, we get $\omega=-(1-\alpha^{2}\Delta)\Delta\phi=-\Delta \psi$, where we denote by $\psi$, the stream function associated with velocity $\bu$, i.e., $\bu=-\nt\psi$ and $\phi$ and $\psi$ are related via the formula $\psi=(1-\alpha^{2}\Delta)\phi$. We can thus rewrite \eqref{EEV_alpha} in terms of the stream function variables as 
\begin{equation}\lb{EES_alpha}
\Delta\psi_t-\nt\phi\cdot\nabla(\Delta\psi)=\Delta\psi_t-\phi_x\Delta\psi_y+\phi_y\Delta\psi_x=0.
\end{equation}
 
We have the following version of the Kelvin circulation theorem for the $\alpha$ Euler model. The proof is omitted as it is similar to the proof for the Euler case which can be found, for instance, in \cite[pp.21-22]{CM93}.
\begin{lemma}\label{Kelvin_Circ_Thm_Alpha}
Let $C$ be a simple closed contour in the fluid at time $t=0$. Let $C_{t}$ be the contour carried along by the flow following the smoothed velocity $\bv$, i.e $C_{t}= \chi_{t}(C)$, where $\chi_{t}$ is the flow map associated with the velocity $\bv$, where $\bv$ satisfies \eqref{EE_alpha}. The circulation around $C_{t}$ is defined to be 
\begin{equation}\label{circulation_def}
\Gamma_{C_{t}}= \int_{C_{t}} (1-\alpha^{2}\Delta)\bv \cdot d \mathbf s.
\end{equation}
Then $\Gamma_{C_{t}}$ is constant in time. That is, \[\frac{d}{dt}\Gamma_{C_{t}}=0.\] 
\end{lemma}
\begin{remark}
\sloppy
One of the original motivations in adding the additional term $(\nabla \bv )^{T} \bu$ to the $\alpha$-Euler equations \eqref{EE_alpha} was to prove the Kelvin circulation theorem above (see \cite{FHT01}, Section 2, page 507).
\end{remark}
\begin{corollary}\label{Kelvin_boundary}
Let $D \subset \mathbb R^{2}$ be a multiply connected bounded domain, bounded by finitely many number of smooth curves $(\partial D)_{i}$ where $i=0,\ldots, n$. Then the circulation along each connected boundary curve remains constant in time, i.e, for every $0 \leq i \leq n$, we have \[\frac{d}{dt}\int_{(\partial D)_{i}} (1-\alpha^{2}\Delta)\bv \cdot d \mathbf s = \frac{d}{dt}\int_{(\partial D)_{i}} \bu \cdot d \mathbf s = 0.\]
\end{corollary}
We shall have occasion to use the following integration by parts formula for vector valued functions
\begin{equation}\label{int_by_parts}
\int_{D} \bv \cdot \Delta \bv d \bx = \sum_{i=0}^{n}\int_{(\partial D)_{i}} \bv \cdot (\bn \cdot \nabla) \bv ds - \int_{D} | \nabla \bv |^{2} d \bx.
\end{equation}
Here $\bv=(v_{1},v_{2})$, $| \nabla \bv |^{2}= tr(\nabla \bv \cdot (\nabla \bv)^{T}) $, where $(\cdot)^{T}$ denotes the transpose, and $tr$ denotes trace, i.e $| \nabla \bv |^{2} = \sum_{i=1}^{2} [(\partial_{x} v_{i})^{2}+(\partial_{y} v_{i})^{2}] $. Equation \eqref{int_by_parts} follows from Green's identity for scalar valued functions $f$ and $g$, see, e.g., \cite[Theorem 3 (ii), page 712]{E10},
\begin{equation}\label{greensidentity}
\int_{D} f  \Delta g d \bx = \sum_{i=0}^{n}\int_{(\partial D)_{i}} f  (\bn \cdot \nabla) g ds - \int_{D} \nabla f \cdot \nabla g d \bx,
\end{equation} 
applied to each component of $\bv$ and summation of the resulting identities. 
By assumption, for any $\bv $ that satisfies \eqref{eeabc}, $\bn \cdot \nabla \bv$ is parallel to $\bn$ and $\bv \cdot \bn =0$, on the boundary, we have,
 $\bv \cdot (\bn \cdot \nabla \bv)=0$, i.e., we have, 
\begin{equation}\label{int_by_parts_alph_Eul}
\int_{D} \bv \cdot \Delta \bv d \mathbf x = - \int_{D} | \nabla \bv |^{2} d \bx.
\end{equation}
If $\bv_{1}$ and $\bv_{2}$ are two vector fields that satisfy \eqref{eeabc}, then by integrating by parts twice using formula \eqref{int_by_parts}, observing that the boundary terms vanish via \eqref{eeabc}, and using the fact that $tr(\nabla \bv_{1} \cdot (\nabla \bv_{2})^{T})=tr(\nabla \bv_{2} \cdot (\nabla \bv_{1})^{T})$ one also sees that
\begin{equation}\label{v1v2}
\int_{D} \bv_{1} \cdot \Delta \bv_{2} d \mathbf x  = \int_{D} \Delta \bv_{1} \cdot  \bv_{2} d \mathbf x.
\end{equation}

Consider a steady state solution $\omega^0=\curl(1- \alpha^2 \Delta)\bv^0=-\Delta\psi^0$ of the $\alpha$-Euler equation. Here \[\bv^{0}=-\nabla^{\perp}\phi^{0}, \quad \bu^{0}=(1-\alpha^{2}\Delta)\bv^{0}, \quad \psi^{0}=(1-\alpha^{2}\Delta)\phi^{0} \mbox{ and } \omega^{0}=-(1-\alpha^{2}\Delta)\Delta \phi^{0}.\] In particular,
\begin{equation}\lb{stst}\nt\phi^0\cdot\nabla\omega^0=\big(-\phi^0_y\partial_x+\phi^0_x\partial_y\big)\omega^0=0,
\end{equation}
and thus $\nabla\phi^0$ and $\nabla(\Delta\psi^0)$ are parallel. 
The corresponding linearized equations for the $\alpha$-Euler model about the steady state $\bv^{0}$, $\bu^{0}$ are as follows: 
\begin{align}
\bu_t&+\bv^0\cdot\nabla\bu+\bv\cdot\nabla\bu^0+\big(\nabla \bv^0 \big)^{\top} \bu+\big(\nabla \bv \big)^{\top} \bu^0 +\nabla \pi=0, \lb{LEE_alpha}  \\&
 \div\bu=\div\bv=0, \nonumber\\
\omega_t&+\bv^0\cdot\nabla\omega+\bv \cdot\nabla\omega^0=0,\lb{LEEV_alpha}\\
\Delta\psi_t&-\phi^0_x\Delta\psi_y+\phi^0_y\Delta\psi_x-\phi_x\Delta\psi^0_y+\phi_y\Delta\psi^0_x=0,\lb{LEES_alpha}
\end{align}
where $\bv=\curl^{-1}(1-\alpha^{2}\Delta)^{-1}\omega$ solves the system of equations $\curl (1-\alpha^{2}\Delta)\bv=\omega$, $\div \bv=0$ with appropriate boundary conditions.

\section{Rayleigh and Fjortoft criteria for the $\alpha$-Euler equations}
In this section, we derive the classical Rayleigh and Fjortoft criteria for the $\alpha$-Euler equations, see \cite[Sec. 22]{DZ04} for an overview of these results for the Euler equations.  
Our basic setup is the two dimensional channel, infinitely long in the $x$ direction and bounded in the $y$ direction, with walls at $y=A_{1}$ and $y=A_{2}$, where $-\infty < A_{1} < A_{2}<+\infty$, i.e., $ D = \mathbb R \times [A_{1},A_{2}]$. 

We work with $2D$ plane parallel shear flows in the channel. Note that for the two dimensional Euler equations on the domain $D$, any steady state velocity of the form $\bu^0(x,y)= (U(y),0)$ and constant pressure $p(x,y)=p^0$ will solve the Euler equations \eqref{EE}, where $U:[A_{1},A_{2}] \to \mathbb R$ is any real valued smooth function and $p^0$ is a real constant. Note that the boundary condition $\bu \cdot \bn=0 $ is automatically satisfied by a steady state of the type $(U(y),0)$.

We now obtain the analogous steady state for the $\alpha$-Euler equations \eqref{EE_alpha}. Let $V:[A_{1},A_{2}] \to \mathbb R$ be a  real valued smooth function such that $V'(A_{1})=V'(A_{2})=0$. Define $U(y) = V(y) - \alpha^2 V''(y)$. Then $(U(y),0)$ as computed and $p^0$ are steady state solutions of the Euler equations \eqref{EE}. 
It can be readily verified (see \cite{LNTZ15}, Prop. 2, pp. 60) that $\bu^0(x,y)=(U(y),0)$, $\bv^0(x,y)=(V(y),0)$ and 
\begin{equation}\label{pressuresteadyst}
\pi(x,y)=p_0 - \frac{1}{2}V(y)^2+\frac{\alpha^2} {2}(V'(y))^2
\end{equation}
 (note that $\pi$ is a function of $y$ alone) is a steady state solution to the $\alpha$-Euler equations \eqref{EE_alpha}, where $\bu^{0}=(1-\alpha^{2}\Delta)\bv^{0}$. Since $\bn=(0,\pm1)$ on the boundaries $y=A_{1}$ and $y=A_{2}$, $(\bn \cdot \nabla) \bv^{0}= \partial_{y}(V(y),0)=(V'(y),0)$. Since $(\bn \cdot \nabla) \bv^{0} \cdot \bt =(V'(y),0)\cdot (1,0)= 0$ on the boundary $y=A_{1}$ and $y=A_{2}$, this reduces to $V'(A_{1})=V'(A_{2})=0$. Another way to obtain a steady state is to start with an arbitrary profile $U(y)$ and compute $V(y)$ by solving the ODE:
\begin{equation}\label{vode}
V(y)-\alpha^{2}V''(y)=U(y), \quad V'(A_{1})=V'(A_{2})=0. 
\end{equation}
\begin{lemma}\label{alphaeulersteadyst}
\sloppy
Let $V(y)$ be any smooth function satisfying $V'(A_{1})=V'(A_{2})=0$, and define $U(y)=V(y)-\alpha^{2}V''(y)$, $\pi(x,y)=p_0 - \frac{1}{2}V(y)^2+\frac{\alpha^2} {2}(V'(y))^2$. Then $\bu^0(x,y)=(U(y),0)$, $\bv^0(x,y)=(V(y),0)$ and $\pi(x,y)=p_0 - \frac{1}{2}V(y)^2+\frac{\alpha^2} {2}(V'(y))^2$ is a steady state solution to the $\alpha$-Euler equations \eqref{EE_alpha} on the domain $\mathbb R \times [A_{1},A_{2}]$. 
\end{lemma}
\begin{remark}
We stress that an arbitrary profile $V(y)$ cannot be a steady state for the $\alpha$-Euler equations. Only profiles that satisfy the boundary condition $V'(A_{1})=V'(A_{2})=0$ can be steady states in contrast to the Euler case where an arbitrary profile $U(y)$ with no extra boundary conditions is a steady state for Euler.
\end{remark}
We let $\psi^0$ and $\phi^0$ be the steady state stream functions, real valued, associated with the respective steady state velocities $U$ and $V$ respectively, i.e we have, $\partial_{y} \phi^{0}(y)=V(y)$ and $\partial_{y} \psi^{0}(y)=U(y)$ and $\psi^{0}=(1-\alpha^{2}\partial_{yy})\phi^{0}$. Note that $\psi^0$ and $\phi^0$ are functions of $y$ alone. The boundary condition $V'(A_{1})=V'(A_{2})=0$ implies that $\phi^{0}_{yy}(A_{1})=\phi^{0}_{yy}(A_{2})=0$. We work with the $\alpha$-Euler equations in stream function formulation \eqref{LEES_alpha}. Linearizing \eqref{LEES_alpha} about the steady state $\psi^0,\phi^0$ we obtain the linearized equation for the perturbations $\widetilde{\phi}=\widetilde{\phi}(x,y,t)$ and $\widetilde{\psi}=\widetilde{\psi}(x,y,t)$ of the stream function \eqref{LEES_alpha} of the form
\begin{equation}\label{esfa}
 \Delta\widetilde{\psi}_t -\phi^0_x\Delta\widetilde{\psi}_y+\phi^0_y\Delta\widetilde{\psi}_x-\widetilde{\phi}_x\Delta\psi^0_y+\widetilde{\phi}_y\Delta\psi^0_x=0.
\end{equation}
Since we have the relation $(1-\alpha^{2}\Delta)\widetilde{\phi}=\widetilde{\psi}$ one can consider the equation above as an equation for $\widetilde{\phi}$ alone. We note that this equation is supplemented by the boundary conditions: no normal flow across the boundaries, so $\bv\cdot \bn = 0$ on the boundary $\partial D$, which are the two walls at $y=A_{1}$ and $y=A_{2}$, and $\bn$ is the unit normal vector on $\partial D$ and $(\bn \cdot \nabla )\bv$ is parallel to $\bn$. Since, $\bv= -\nabla^{\perp}\widetilde{\phi}$, we see that the boundary conditions for $\widetilde{\phi}$ are $\nabla \widetilde{\phi} \cdot \bt = 0$ on $\partial D$ where $\bt$ is the unit tangent vector on $\partial D$ and $(\bn \cdot \nabla )(-\nabla^{\perp}\widetilde{\phi}) \cdot \bt = 0$ on $\partial D$. 

Since $\phi^0_x=\psi^0_x=0$, equation \eqref{esfa} reduces to 
\begin{equation} \label{lin_strm_fcn}
\Delta\widetilde{\psi}_t +\phi^0_y\Delta\widetilde{\psi}_x-\widetilde{\phi}_x\Delta\psi^0_y=0,
\end{equation}
on $D$ with $\nabla \widetilde{\phi} \cdot \bt = 0$ and $(\bn \cdot \nabla )(-\nabla^{\perp}\widetilde{\phi}) \cdot \bt = 0$ on $\partial D$.
Similar to the analysis for the Euler equations, see \cite[Sec. 20-22, pp. 124-133]{DZ04}, we look for solutions to \eqref{lin_strm_fcn} of the form 
\begin{equation}\label{stranzatz}
\widetilde{\psi}(x,y,t)=\psi(y)e^{ik(x-ct)} \mbox{ and } \widetilde{\phi}(x,y,t)=\phi(y)e^{ik(x-ct)},
\end{equation}
where $\psi:[A_{1},A_{2}] \to \mathbb C$ and $\phi:[A_{1},A_{2}] \to \mathbb C$ are complex valued functions of $y$, $k \in \mathbb R$ is the wave number, which is real in this case and $c$ is the wave speed which is complex valued, $c=c_r+ic_i$. If $c_{i}>0$, this corresponds to an exponentially growing (in time) solution to \eqref{lin_strm_fcn}.

We will now derive an $\alpha$-Euler version of the Rayleigh stability equation. 
\begin{lemma}\label{rst}
Let $\bu^0(x,y)=(U(y),0)$, $\bv^0(x,y)=(V(y),0)$ and $\pi(x,y)=p_0 - \frac{1}{2}V(y)^2+\frac{\alpha^2} {2}(V'(y))^2$ be a steady state solution to the $\alpha$-Euler equations \eqref{EE_alpha} on the domain $\mathbb R \times [A_{1},A_{2}]$. Suppose the linearized stream function equation \eqref{lin_strm_fcn} about this steady state has a solution of the form \eqref{stranzatz} with $c_{i}>0$, and some $k \in \mathbb R$. Then $\phi:[A_{1},A_{2}] \to \mathbb C$ satisfies the following boundary value problem:
\begin{equation} \label{RSE_alpha}
\begin{split}
-\alpha^{2}\phi'''' +(1+2\alpha^{2} k^{2})\phi''- k^{2}(1+\alpha^{2}k^{2})\phi-\frac{U'' \phi}{V-c}=0\\
\phi(A_{1})=\phi(A_2)=0 \\
\phi''(A_{1})=\phi''(A_2)=0.
\end{split}
\end{equation}
\end{lemma}
\begin{proof}
Since $\widetilde{\psi}$ and $\widetilde{\phi}$ are related via $(1-\alpha^2 \Delta) \widetilde{\phi}=\widetilde{\psi}$, we have the following relation between $\psi$ and $\phi$,
\begin{equation} \label{amp_reln}
\psi(y)=(1+\alpha^2k^2) \phi(y)-\alpha^2 \phi''(y).
\end{equation}
Using \eqref{amp_reln}, $\widetilde{\phi}(x,y,t)=\phi(y)e^{ik(x-ct)}$, $\widetilde{\psi}(x,y,t)=\psi(y)e^{ik(x-ct)}$, $\partial_{y} \phi^{0}(y)=V(y)$ and $\partial_{y} \psi^{0}(y)=U(y)$,  
we see that \eqref{lin_strm_fcn} yields

\begin{equation}\label{computations_2}
 (\psi''-k^{2}\psi)(V-c)+U''\phi=0. 
\end{equation}
The boundary conditions for $\widetilde{\phi}$ are $\nabla \widetilde{\phi} \cdot \bt = 0$ on $\partial D$, where the boundary corresponds to $y=A_{1}$ and $y=A_{2}$. Since $\bt=(1,0)$ on $\partial D$, this means that $\partial_{x}\widetilde{\phi}=0$ at $y=A_{1}$ and $y=A_2$. Since $\partial_{x}\widetilde{\phi}(x,y,t)=ik \phi(y)e^{ik(x-ct)}=0$, when $y=A_{1}$ and $y=A_2$, we have the first boundary condition in \eqref{RSE_alpha}.
The boundary condition corresponding to $\bn \cdot \nabla (-\nabla^{\perp}\widetilde{\phi}) \cdot \bt =0$ is computed as follows: we note that $\bn \cdot \nabla = (0,1) \cdot (\partial_{x}, \partial_{y})= \partial_{y}$. Thus $\bn \cdot \nabla (-\nabla^{\perp}\widetilde{\phi}) = \partial_{y} (\partial_{y} \widetilde{\phi}, -\partial_{x}\widetilde{\phi})= (\widetilde{\phi}_{yy},\widetilde{\phi}_{xy})$. Thus, $\bn \cdot \nabla (-\nabla^{\perp}\widetilde{\phi}) \cdot \bt$ becomes $(\widetilde{\phi}_{yy},\widetilde{\phi}_{xy}) \cdot (1,0)=(\widetilde{\phi}_{yy},0)$. Thus, we get that, $\bn \cdot \nabla (-\nabla^{\perp}\widetilde{\phi}) \cdot \bt =0$ becomes $\widetilde{\phi}_{yy}(x,y,t)=0$ when $y=A_{1}$ and $y=A_{2}$. Since, $\widetilde{\phi}_{yy}(x,y,t)=\phi''(y)e^{ik(x-ct)}$, this corresponds to the second boundary condition in \eqref{RSE_alpha}.
Using \eqref{amp_reln} we note that \eqref{computations_2} yields \eqref{RSE_alpha}.
\end{proof}
The analogous equation for the Euler equation is the classical Rayleigh stability equation, known as the inviscid Orr-Sommerfeld equation \cite[pp. 122, Eq. 4.6]{MP94}. 

\begin{remark}\label{functionspace}
Note that $\widetilde{\phi}(x,y,t)=\phi(y)e^{ik(x-ct)}=\phi(y)e^{ikx}e^{-ikct}=\widehat{\phi}(x,y)e^{-ikct}$, where $\widehat{\phi}(x,y)=\phi(y)e^{ikx}$. The smoothness of $\widetilde{\phi}$ depends on the smoothness of $\phi$. If we consider $\phi$ to be in the space \[Y:=\{\phi \in H^{4}([A_{1},A_2]; \mathbb C), \phi(A_{1})=\phi(A_{2})=0, \phi''(A_{1})=\phi''(A_{2})=0\},\] to solve \eqref{RSE_alpha}, then we obtain a  solution to \eqref{lin_strm_fcn} and we measure instability for \eqref{lin_strm_fcn} in the space \[\widehat{Y}:= \{\widehat{\phi}(\cdot,\cdot):\mbox{sup}_{x \in \mathbb R}||\widehat{\phi}(x,\cdot) ||_{H^{4}([A_{1},A_{2}]; \mathbb C)} < \infty \}.\] 
We shall consider \eqref{lin_strm_fcn} as an (linear) evolution equation for $\widetilde{\phi}=\widetilde{\phi}(\cdot,\cdot,t)$ in the space $\widehat{Y}$, and obtain conditions for existence of a growing eigenmode solution, i.e for spectral instability to the linearized equation \eqref{lin_strm_fcn} for the perturbation stream function. All references to instability in this subsection is to be regarded in the sense described above. 
A solution of the form $\widetilde{\phi}(x,y,t)=\phi(y)e^{ik(x-ct)}$, with $\phi \in Y$, with $c_{i}> 0$ to equation \eqref{lin_strm_fcn} grows exponentially in time. This corresponds to spectral instability for \eqref{lin_strm_fcn} in space $\widehat{Y}$ with spectral parameter $\lambda=-ikc$. 
\end{remark}
We are ready to prove an analogue of the classical Rayleigh's theorem for the $\alpha$-Euler equations.
\begin{proposition}\label{prop_Rayl_crit_alpha}
\sloppy
Let $\bu^0(x,y)=(U(y),0)$, $\bv^0(x,y)=(V(y),0)$ and $\pi(x,y)=p_0 - \frac{1}{2}V(y)^2+\frac{\alpha^2} {2}(V'(y))^2$ be a steady state solution to the $\alpha$-Euler equations \eqref{EE_alpha} on the domain $\mathbb R \times [A_{1},A_{2}]$. If the linearized stream function equation \eqref{lin_strm_fcn} about this steady state has solutions of the form \eqref{stranzatz} with $c_{i}>0$, then $U''(y_{s}) = 0$ for at least one point $y_{s} \in [A_{1},A_{2}]$.
\end{proposition}
\begin{proof}
By assumption $c_i > 0$, we have $V-c \neq 0$ because $V$ is real valued and thus \eqref{RSE_alpha} is non-singular. We multiply the first equation of \eqref{RSE_alpha} by $\phi^*$, the complex conjugate of $\phi$, integrate by parts (using the boundary conditions, i.e the second and third equations of \eqref{RSE_alpha}) to obtain, 
\begin{multline*}
-\alpha^{2}\int_{A_1}^{A_2} |\phi''(y)|^{2}dy -  \int_{A_1}^{A_2}(1+2\alpha^{2} k^{2})|\phi'(y)|^{2}dy - \int_{A_1}^{A_2}k^{2}(1+\alpha^{2}k^{2})|\phi(y)|^{2} dy \\
-\int_{A_1}^{A_2} \frac{U''(y)}{V(y)-c}|\phi(y)|^{2} dy = 0. \numberthis \label{int_RSE}
\end{multline*}
Indeed, the boundary terms vanish via the boundary conditions. 
Taking the imaginary part of \eqref{int_RSE} gives 
\begin{equation} \label{R_criterion}
c_i \int_{A_{1}}^{A_{2}} \frac{U''(y)}{|V(y)-c|^2}|\phi(y)|^2 dy = 0
\end{equation}
Since $c_{i}>0$, this forces $U''(y_{s})=0$ for at least one point $y_{s} \in [A_{1},A_{2}]$.
\end{proof}
We thus have that if $U=V-\alpha^{2}V''$ does not have an inflection point in $[A_{1},A_{2}]$ then \eqref{lin_strm_fcn} cannot have an exponentially growing eigensolution of the form $\phi(y)e^{ik(x-ct)}$. 
\begin{example}\label{rex}
 Let $U(y)$ be a profile with no inflection point. Compute $V(y)$ by solving the ODE: $V(y)-\alpha^{2}V''(y)=U(y)$ subject to the boundary conditions $V'(A_{1})=V'(A_{2})=0$. The steady state thus computed, with the pressure given by equation \eqref{pressuresteadyst}, cannot be spectrally unstable by Proposition \ref{prop_Rayl_crit_alpha}. This class of steady states include profiles $U(y)$ that are symmetric about the center point with no inflection points, for example, $U(y)=1-y^{2}$ on the interval $[-1,1]$ and it also includes linear steady states such as the Couette flow $U(y)=y$ on  $[0,1]$. 
\end{example}
We now derive an analogue of the classical Fjortoft's theorem for the $\alpha$-Euler equations.
\begin{proposition}\label{prop_F_crit_alpha}
\sloppy
Let $\bu^0(x,y)=(U(y),0)$, $\bv^0(x,y)=(V(y),0)$ and $\pi(x,y)=p_0 - \frac{1}{2}V(y)^2+\frac{\alpha^2} {2}(V'(y))^2$ be a steady state solution to the $\alpha$-Euler equations \eqref{EE_alpha} on the domain $\mathbb R \times [A_{1},A_{2}]$. If the linearized stream function equation \eqref{lin_strm_fcn} about this steady state has solutions of the form \eqref{stranzatz} with $c_{i}>0$, then $U''(y)(V(y)-V(y_{s})) < 0$ for at least one point $y \in [A_{1},A_{2}]$. Here $y_{s} \in [A_{1},A_{2}]$ is a point such that $U''(y_{s})=0$ obtained in Proposition \ref{prop_Rayl_crit_alpha}.
\end{proposition}
\begin{proof}
Let $V_s=V(y_s)$. Consider the real part of \eqref{int_RSE} 
and adding
\begin{equation}\label{add}
(c_r-V_s) \int_{A_1}^{A_2} \frac{U''(y)}{|V(y)-c|^2}|\phi(y)|^2 dy = 0,
\end{equation}
we obtain,
\begin{multline}\label{F_crit}
\int_{A_1}^{A_2} \frac{U''(y)(V(y)-V_s)}{|V(y)-c|^2}|\phi(y)|^2 dy  = -\alpha^{2}\int_{A_1}^{A_2} |\phi''(y)|^{2}dy  \\
-\int_{A_1}^{A_2}(1+2\alpha^{2} k^{2})|\phi'(y)|^{2}dy -\int_{A_1}^{A_2}k^{2}(1+\alpha^{2}k^{2})|\phi(y)|^{2} dy < 0,
\end{multline}
since the integrands on the right hand side are non negative. %
Thus $U''(y)(V(y)-V(y_{s})) < 0$ for at least one point $y \in [A_{1},A_{2}]$. 
\end{proof}
\begin{remark}\label{vs}
Note that in the proof of the above theorem since the integral in \eqref{add} is zero, the coefficient in front of the integral in \eqref{add} can be replaced by any number. The statement of Fjortoft criterion can be generalized to the following fact: for every real number $z$, a necessary condition for instability is the existence of at least one point $y \in [A_{1},A_{2}]$ such that
\begin{equation}\label{fgeneral}
(V(y)-z)U''(y) < 0.
\end{equation}
In fact, it is readily seen that the Fjortoft criterion implies the Rayleigh criterion by choosing points $z_{1}$ and $z_{2}$ such that $V(y)-z_{1}>0$ and $V(y)-z_{2}<0$ for all $y \in [A_{1},A_{2}]$. Then Fjortoft criterion says that there exists at least two points $y_{1}$ and $y_{2}$, such that $U''(y_{1}) < 0$ and $U''(y_{2}) > 0$ thereby ensuring that $U''(y)$ changes sign somewhere in the flow. 
\end{remark}
\begin{remark}\label{fexample}
We recall that Fjortoft criterion for the Euler equation says that a necessary condition for instability is that $U''(y)(U(y)-U(y_{s}))  < 0$ for at least one point $y \in [A_{1},A_{2}]$. For the Euler equations, Fjortoft criterion is sharper than the Rayleigh criterion. The Fjortoft criterion was used to study steady states $U(y)$ with a monotone profile and one inflection point $y_{s}$.  If $U''$ and $U-U(y_{s})$ have the same sign everywhere in the flow, then $U(y)$ is a stable steady state even though it has an inflection point $y_{s}$. The analogous result for $\alpha$-Euler is as follows. Suppose we have a profile $U(y)$ which is monotone and has one inflection point $y_{s}$. We compute $V(y)$ using equation \eqref{vode}. If $U''(y)$ and $V(y)-V(y_{s})$ have the same sign everywhere, then the steady state is stable for $\alpha$-Euler.
\end{remark}
\begin{example}\label{funstable}
Let us consider $V(y)=y-y^{3}$ on the interval $[-1/\sqrt{3},1/\sqrt{3}]$. Note that $V$ is monotone with one inflection point at $y=0$ and $V'(-1/\sqrt{3})=V'(1/\sqrt{3})=0$. One can compute $U(y)=V(y)-\alpha^{2}V''(y)=y-y^{3}+6\alpha^{2}y$. One can see that $U''(y)=-6y$. $U(y)$ thus has one inflection point at $y_{s}=0$. Note that $V(0)=0$. Thus $U''(y)(V(y)-V(y_{s}))=-6y(y-y^{3})=-6y^{2}(1-y^{2}) < 0$ since $(1-y^{2})$ is everywhere positive in the interval $[-1/\sqrt{3},1/\sqrt{3}]$. Proposition \ref{prop_F_crit_alpha} says that this steady state, with pressure as in \eqref{pressuresteadyst} is possibly unstable for the $\alpha$-Euler equations.
\end{example}

\section{Arnold stability theorems for $\alpha$-Euler equations}
In the 1960's, V.\ I.\ Arnold in \cite{A65} introduced a simple and beautiful idea to study nonlinear Lyapunov stability of ideal fluids. It relied on exploiting the underlying Hamiltonian structure that the fluid model possessed together with convexity estimates on the second variation. For an overview of the Arnold criterion for the Euler equation, see, for example, \cite{A65}, \cite{MP94} (Section 3.2, pp. 104-111), \cite{AK98} (Chapter 2, Section 4, pp. 88-94), \cite{HMRW85} for a more mathematical perspective, \cite{T10}, (Section 1), \cite{B15} (Section 3, page 7), \cite{S00} (Section 4.5, pp. 114-122) for a more applied perspective. Expanded later into the so called energy-Casimir method, this has spawned a huge literature and has been applied widely to study the stability of various model fluid equations. For a non exhaustive sample, see for example \cite{CSS06, GW95, GW96, HMH14,HMR98,  MPSW01, MW01, WG98} as well as the literature cited therein. To the best of our knowledge, however, Arnold's stability theorems were not recorded for the $\alpha$-Euler model, and in this section we fill the gap. 
\subsection{Arnold's theorems in a multi connected domain}
Let $D \subset \mathbb R^{2}$ be a bounded multi connected domain, with the boundary $\partial D$ consisting of smooth curves $(\partial D)_{i}$ where $i=0,\ldots, n$. Denote the outer boundary of $D$ by $(\partial D)_0$ and $(\partial D)_{i}, i=1,\ldots,n$ are the $n$ inner boundaries. Let $\bv^{0}$ denote a steady state solution of \eqref{EEV_alpha}, $\phi^{0}$ denote its stream function, so that $\bv^{0}=-\nabla^{\perp}\phi^{0}$ and $\omega^{0}$ denote its vorticity, so that \[\omega^{0}=\curl (1-\alpha^{2}\Delta)\bv^{0}=-\Delta(1-\alpha^{2}\Delta) \phi^{0}.\]  
Since $\nabla^{\perp}\phi^{0}\cdot \nabla \omega^{0}=0$ by \eqref{EEV_alpha}, we have that $\nabla \phi^{0}$ is parallel to $\nabla \omega^{0}$ and thus locally, $\phi^{0}$ is a function of $\omega^{0}$. We shall impose the following global condition.
\begin{assumption}\label{steadystassumptionalpha}
Assume that there exists a differentiable function $F$ defined on the closed interval $\left[\mbox{min }_{(x,y) \in \overline{D}} \omega^{0}
(x,y),\mbox{max }_{(x,y) \in \overline{D}} \omega^{0}
(x,y)\right]$, such that, $\phi^{0}(x,y)=F(\omega^{0}(x,y))$ for every $(x,y) \in D$.
\end{assumption}
In particular, we have that $\nabla \phi^{0}(x,y)=F'(\omega^{0}(x,y))\nabla \omega^{0}(x,y)$ and thus $ \bv^{0}=-\nabla^{\perp}\phi^{0}=-F'(\omega^{0})\nabla^{\perp} \omega^{0}$. 
\begin{remark}
We note that sometimes, in the literature, see, for example, \cite{MP94}, (page 106, Remark 1) the function $-F'$ is written as the ratio of the vectors $\frac{\bv^{0}}{\nabla^{\perp}\omega^{0}}$.
\end{remark}
Notice that, $F'$, apriori, can have singularities at the critical points of $\omega^{0}$. Since $\bv \cdot \bn=0$ on the boundary $\partial D$, we have, on $\partial D$, that $\nabla^{\perp}\phi \cdot \bn =0$, i.e., $\nabla \phi \cdot \bt =0$, where $\bt$ is the unit tangent vector on the boundary. Thus $\nabla \phi$ is orthogonal to the boundary, i.e., each connected piece of the boundary curve is a level set of $\phi$ and thus $\phi|_{(\partial D)_{i}}$ is a constant on each connected boundary piece $(\partial D)_{i}$ where $0 \leq i \leq n$.
Assumption \ref{steadystassumptionalpha} then implies that $F(\omega^{0})$ restricted to each connected piece $(\partial D)_{i}$, $0 \leq i \leq n$, is a constant. We shall work on the following subspace of the Sobolev space 
\begin{align}\label{x}
X:= \bigg\{ \bv \in H^{3}(D;\mathbb R^{2}), & \nabla \cdot \bv =0  \mbox{ on } D, \nonumber \\& \bv \cdot \bn = 0,
(\bn \cdot \nabla \bv) \parallel \bn , \mbox{on } (\partial D)_{i}, 0 \leq i \leq n \bigg\},
\end{align}
so that $\bu \in H^{1}(D;\mathbb R^{2})$, $\omega=\curl \bu \in L^{2}(D;\mathbb R)$, $\phi \in H^{4}(D;\mathbb R)$ and $\psi \in H^{2}(D;\mathbb R)$. The analysis proceeds by considering the following functional $H_{c}:X \to \mathbb R$ 
\begin{equation} \label{hc}
H_{c}(\bv)= H(\bv)+Q(\bv)+ \sum_{i=0}^{n} a_{i} \int_{(\partial D)_{i}} (1-\alpha^{2}\Delta)\bv \cdot d \mathbf s,
\end{equation}
and $a_{i} \in \mathbb R$ for $i=0,\ldots,n$ where $H,Q:X \to \mathbb R$ on $X$ are given by 
\begin{align}\label{hq}
H(\bv)&=\frac{1}{2}\int_{D} \bv(\bx) \cdot (1-\alpha^{2}\Delta)\bv(\bx) d\bx,\\
Q(\bv)&=\int_{D} C(\curl (1-\alpha^{2}\Delta)\bv(\bx))d\bx, 
\end{align}
and a smooth function $C:\mathbb R \to \mathbb R$ will be fixed later. We note that $H(\bv)$ is the kinetic energy of the $\alpha$-Euler fluid, $C$ is known as the Casimir function, the last term in \eqref{hc} corresponds to the weighted sum of circulations along the boundary.
\begin{lemma}\label{H_inv}
Let $\bv=\bv(t,\bx)$ solve the $\alpha$-Euler equations \eqref{EE_alpha}. Then
\begin{align}\label{hc2}
H_{c}(\bv(t,\cdot))= & \nonumber
\frac{1}{2}\int_{D} \bv(t,\bx) \cdot (1-\alpha^{2}\Delta)\bv(t,\bx) d\bx + \int_{D} C(\curl (1-\alpha^{2}\Delta)\bv(t,\bx))d\bx
\\ \nonumber
&+ \sum_{i=0}^{n} a_{i} \int_{(\partial D)_{i}} (1-\alpha^{2}\Delta)\bv(\mathbf s,t) \cdot d \mathbf s, \numberthis 
\end{align}
is an invariant of motion, that is, $\frac{d}{dt}H_{c}(\bv(t,\cdot))=0$.
\end{lemma}
\begin{proof}
We note that 
$\bu=(1-\alpha^{2}\Delta)\bv$ and $\omega=\curl \bu = \curl (1-\alpha^{2}\Delta)\bv$.
We first show that $\frac{d}{dt}H(\bv(t,\cdot)) =0$. 
We start with the $\alpha$-Euler equations \eqref{EE_alpha},
\begin{equation}\label{alph_Eul_vel}
\partial_t\bu + \bv \cdot \nabla \bu + (\nabla \bv)^{\top} \bu = - \nabla \pi,
\end{equation}
and use the following identity (\cite[Eq. 7.34]{HMR98}) \[(\bv \cdot \nabla) \bu + (\nabla \bv)^{\top} \bu = - \bv \times (\nabla \times \bu) + \nabla (\bu \cdot \bv),\] to rewrite \eqref{alph_Eul_vel} as 
\begin{equation} \label{alph_vel_new}
\partial_t\bu - \bv \times (\nabla \times \bu) + \nabla (\bu \cdot \bv) = - \nabla \pi. 
\end{equation} Taking the dot product of \eqref{alph_vel_new} with $\bv$, noting that $- [\bv \times (\nabla \times \bu)]  \cdot \bv =0$, and integrate over the domain $D$ to get,
\begin{equation}\label{KE_identity}
 \int_{D}\bu_{t}(\bx) \cdot \bv(\bx)d \bx = 0,
\end{equation} 
where we used the facts that $\int_{D} \nabla (\bu \cdot \bv) \cdot \bv d \bx =0$ and $\int_{D} \nabla \pi \cdot \bv d \bx = 0$. Indeed, for any scalar valued function $f$, we have 
\begin{equation}\label{alpha5}
(\nabla f) \cdot \bv = \nabla \cdot (f \bv) - f \nabla \cdot \bv = \nabla \cdot (f \bv),
\end{equation} 
(because $\bv$ is divergence free, $\nabla \cdot \bv =0$) and by \eqref{eeabc} and Divergence Theorem, 
\begin{equation}\label{Div_thm}
\int_{D} \nabla \cdot (f \bv)d \bx = \sum_{i=0}^{n}\int_{(\partial D)_{i}} f \bv \cdot \bn ds = 0.
\end{equation}
We rewrite $H(\bv)$ as:
\begin{align*}
H(\bv)= &\frac{1}{2}\int_{D}\bu\cdot \bv d \bx = \frac{1}{2}\int_{D}(1-\alpha^{2}\Delta)\bv\cdot \bv d \bx \\&
= \frac{1}{2}\int_{D}\bv \cdot \bv d \bx + \frac{\alpha^{2}}{2} \int_{D} |\nabla \bv|^{2}  d \bx \numberthis \label{H}
\end{align*} 
where in the second equality, we performed integration by parts on the second term and use \eqref{int_by_parts_alph_Eul}. Equation \eqref{eeabc} implies that $\bv \cdot ((\bn\cdot \nabla)\bv)=0$. Using \eqref{eeabc} it can also be shown that
\begin{equation}\label{eeabct}
\bv_{t}\cdot \bn=0, \quad \bv_{t} \cdot ((\bn\cdot \nabla)\bv) =0, \quad \bv \cdot ((\bn\cdot \nabla)\bv_{t})=0.
\end{equation}
Indeed, $0=\partial_{t}(\bv \cdot \bn)=\bv_{t} \cdot \bn$ and using this and the second equation in \eqref{eeabc}, we see that $\bv_{t} \cdot ((\bn\cdot \nabla)\bv) =0$. Now use the fact that $\partial_{t}(\bv \cdot (\bn\cdot \nabla\bv))=0$ and $\bv_{t} \cdot ((\bn\cdot \nabla)\bv) =0$ to conclude the third equality in \eqref{eeabct}.
Recall that $|\nabla \bv|^{2}=tr(\nabla \bv \cdot (\nabla \bv)^{T})$. Also notice that \[\frac{1}{2}\partial_{t} tr(\nabla \bv \cdot (\nabla \bv)^{T}) = tr(\nabla \bv_{t} \cdot (\nabla \bv)^{T})\] and integration by parts using the boundary conditions \eqref{eeabct} yields 
\begin{equation}\label{tr}
\frac{1}{2}\int_{D} \partial_{t} tr(\nabla \bv \cdot (\nabla \bv)^{T}) d \bx = \int_{D} tr(\nabla \bv_{t} \cdot (\nabla \bv)^{T}) d \bx = - \int_{D} \Delta \bv_{t} \cdot \bv d \bx.
\end{equation}
We thus have, by \eqref{KE_identity}, \eqref{H} and \eqref{tr}, 
\begin{align*}
\frac{d}{dt}H(\bv(t,\cdot)) 
&=\frac{1}{2} \int_{D} (\bv(\bx,t)\cdot \bv(\bx,t))_{t} d \bx + \frac{\alpha^{2}}{2} \int_{D}  \partial_{t}|(\nabla \bv(\bx,t)|^{2}  d \bx \\
&= \int_{D} \bv_{t}(\bx,t)\cdot \bv(\bx,t) d \bx +\frac{\alpha^{2}}{2} \int_{D} \partial_{t} tr(\nabla \bv(\bx,t) \cdot (\nabla \bv(\bx,t))^{T})  d \bx\\
&= \int_{D} \bv_{t}(\bx,t)\cdot \bv(\bx,t) d \bx -\alpha^{2} \int_{D} \Delta \bv_{t}(\bx,t) \cdot \bv(\bx,t) d \bx \\&= \int_{D} \big(\bv_{t}(\bx,t) - \alpha^{2}\Delta \bv_{t}(\bx,t)\big)\cdot \bv(\bx,t) d \bx = \int_{D} \bu_{t}(\bx,t) \cdot \bv(\bx,t) d \bx = 0.
\end{align*}
In order to prove that $\frac{d}{dt}\int_{D} C(\omega(\bx,t)) d x = 0$, we use (a slight variant of) the following idea from \cite[see Appendix, page 162]{CSS06}. 
Using \eqref{EEV_alpha} we have, 
\begin{equation}\label{Pvortalphaeuler}
\partial_{t}C(\omega(\bx,t)) = C'(\omega(\bx,t))\partial_{t} \omega(\bx,t) = -C'(\omega(\bx,t))\bv \cdot \nabla\omega(\bx,t),
\end{equation}
and also, \[-\bv \cdot \nabla (C(\omega(\bx,t)))= -C'(\omega(\bx,t))\bv \cdot \nabla\omega(\bx,t).\] Therefore,
\begin{equation}\label{pvortalphaeuler2}
\partial_{t}C(\omega(\bx,t)) = -\bv \cdot \nabla (C(\omega(\bx,t))).
\end{equation}
Using \eqref{alpha5}, \eqref{Div_thm}, \eqref{Pvortalphaeuler} and \eqref{pvortalphaeuler2}, we infer
\begin{align*}
\frac{d}{dt}\int_{D}C(\omega(\bx,t)) d x = \int_{D} \partial_{t} C(\omega(\bx,t)) d x  = \int_{D}  - \bv \cdot \nabla C(\omega(\bx,t)) d x=0.
\end{align*} 
The fact that $\frac{d}{dt} \int_{(\partial D)_{i}} \bu \cdot d \mathbf s=0$, $i=0,\ldots,n$, is a consequence of Corollary \ref{Kelvin_boundary}. Combining these relations proves the lemma.
\end{proof}
Since $\bv^{0}\cdot \nabla \omega^{0}=0$ and $\bv^{0}$ is tangent to the boundary, we have that $\nabla \omega^{0}$ is orthogonal to the boundary, which implies that $\omega^{0}$ is constant on the boundary $(\partial D)_{i}$, $i=0,\ldots,n$, i.e., we shall denote by $\omega^{0}\big|_{(\partial D)_{i}}$ the value of $\omega^{0}$ on the boundary $(\partial D)_{i}$. This also implies that $C'(\omega^{0}(x,y))$ is constant for all $(x,y) \in (\partial D)_{i}$ and we shall denote this by $C'(\omega^{0})\big|_{(\partial D)_{i}}$. 
We return to \eqref{hc}. We will now specify $C$ and $a_{j}$, such that the first variation $\delta H_{c}(\bv^{0})\delta \bv := \frac{d}{d\epsilon}H_{c}(\bv^{0}+\epsilon\delta \bv)|_{\epsilon=0}$ is zero. 
\begin{lemma}\label{critical pt conditions}
Let $\bv^{0}$, $\omega^{0}$ be a steady state solution of \eqref{EE_alpha}, satisfying Assumption \ref{steadystassumptionalpha}, where $\omega^{0}=\curl(1-\alpha^{2}\Delta)\bv^{0}$. 
Let $C$ be a smooth function  so that 
\begin{equation}\label{Cg}
C'(\omega^{0}(x,y))=-F(\omega^{0}(x,y)),
\end{equation} for every $(x,y) \in D$.  
Let $a_{i}=F(\omega^{0})\big|_{(\partial D)_{i}}$, $i=0,\ldots,n$. Then $\delta H_{c}(\bv^{0})\delta \bv=0$, i.e $\bv^{0}$ is a critical point of $H_{c}$. 
\end{lemma}
\begin{proof}
Note first that $H_{c}(\bv)$ can be expressed, using $\bu=(1-\alpha^{2}\Delta)\bv$, as
\begin{equation}
H_{c}(\bv)= \frac{1}{2}\int_{D} \bv \cdot \bu d\bx + \int_{D} C(\omega)d\bx + \sum_{i=0}^{n} a_{i} \int_{(\partial D)_{i}} \bu \cdot d \mathbf s.
\end{equation}
The first variation of $H_{c}$ at $\bv^{0}$ is given by the following expression,
\begin{align}\label{H_first_var}
& \delta H_{c}(\bv^{0})\delta \bv=\nonumber \frac{d}{d \varepsilon}H_{c}(\bv^{0}+\varepsilon \delta \bv)\big|_{\varepsilon=0}\\ \nonumber
&= \frac{1}{2}\int_{D} \bv^{0} \cdot \delta\bu d \bx + \frac{1}{2}\int_{D} \bu^{0} \cdot \delta\bv d \bx 
+\int_{D} C'(\omega^{0}) \delta \omega d \bx + \sum_{i=0}^{n} a_{i} \int_{(\partial D)_{i}} \delta \bu \cdot d \mathbf s, \numberthis
\end{align}
where $\delta\bu=(1-\alpha^{2}\Delta) \delta\bv$ and $\delta \omega = \curl \delta \bu$. 
We will be using the following identity (see \cite[Eq 2.14, page 108]{MP94}), 
\begin{equation}\label{C1}
C'(\omega^{0}) \delta \omega = \curl (C'(\omega^{0})\delta \bu) - C''(\omega^{0})\nabla^{\perp}\omega^{0}\cdot \delta \bu,
\end{equation} 
which follows from the identity 
$\curl (f \bv) = \nabla^{\perp}f \cdot \bv + f \curl \bv$.
Noting that, by Stokes' theorem, 
\begin{equation}\label{C2}
\int_{D} \curl (C'(\omega^{0})\delta \bu)d \bx=\sum_{i=0}^{n}\int_{ (\partial D)_{i}} C'(\omega^{0})\delta \bu \cdot d \mathbf s,
\end{equation} 
we see that, using \eqref{C1} and \eqref{C2}
\begin{align*}
\int_{D} C'(\omega^{0}) \delta \omega d \bx = - \int_{D} C''(\omega^{0})\nabla^{\perp}\omega^{0} \cdot \delta \bu d \bx 
+\sum_{i=0}^{n}\int_{(\partial D)_{i}} C'(\omega^{0} )\delta \bu \cdot d \mathbf s. \numberthis \label{C3}
\end{align*}
We integrate by parts and use \eqref{v1v2}. Thus,
\begin{equation}\label{C4}
 \int_{D} \bu^{0} \cdot \delta \bv d\bx  =  \int_{D} (\bv^{0}- \alpha^{2} \Delta \bv^{0}) \cdot \delta \bv d\bx =\int_{D} \bv^{0} \cdot \delta \bu d\bx.
\end{equation}
Using \eqref{C4} and \eqref{C3}, we see that \eqref{H_first_var} is given by, 
\begin{align}\label{H_first_var_new}
\delta H_{c}(\bv^{0})\delta \bv \nonumber 
&= \int_{D} \bv^{0} \cdot \delta\bu d \bx - \int_{D}  C''(\omega^{0}) \nabla^{\perp}(\omega^{0}) \delta \bu d \bx \\&+ \sum_{i=0}^{n} \int_{(\partial D)_{i}} C'(\omega^{0}) \delta \bu \cdot d \mathbf s + \sum_{i=0}^{n} a_{i} \int_{(\partial D)_{i}} \delta \bu \cdot d \mathbf s. 
\end{align}
Since $\bv^{0}\cdot \nabla \omega^{0}=0$, and $\bv^{0}$ is tangent to the boundary, this means that $\nabla \omega^{0}$ is orthogonal to the boundary and thus $\omega^{0}$ is a constant on the boundary. This then implies that,
\begin{align}\label{H_first_var_new2}
\delta H_{c}(\bv^{0})\delta \bv \nonumber 
&= \int_{D} \bv^{0} \cdot \delta\bu d \bx - \int_{D}  C''(\omega^{0}) \nabla^{\perp}(\omega^{0}) \delta \bu d \bx \\&+ \sum_{i=0}^{n} C'(\omega^{0})|_{(\partial D)_{i}}\int_{(\partial D)_{i}}  \delta \bu \cdot d \mathbf s + \sum_{i=0}^{n} a_{i} \int_{(\partial D)_{i}} \delta \bu \cdot d \mathbf s, 
\end{align}
from which we see that $\delta H_{c}(\bv^{0})\delta \bv=0$ provided, 
\begin{equation}\label{first_var_identity}
\bv^{0}(x,y)=C''(\omega^{0}(x,y))\nabla^{\perp}\omega^{0}(x,y),
\end{equation} %
\begin{equation}\label{first_var_circ}
a_{i}  = - C'(\omega^{0}),\quad i=0,\ldots,n.
\end{equation}
Since, by \eqref{Cg}, $C$ is chosen such that $C'(\omega^{0}(x,y))=-F(\omega^{0}(x,y))$ for every $(x,y) \in D$, then, \[\bv^{0}=-\nabla^{\perp}\phi^{0}=-F'(\omega^{0})\nabla^{\perp}\omega^{0}=C''(\omega^{0})\nabla^{\perp}\omega^{0},\] i.e., \eqref{first_var_identity} holds. Since we have chosen $a_{i}=\psi^{0}|_{(\partial D)_{i}}$ for all $0 \leq i \leq n$, (note that $\psi^{0}$ is a constant on the boundary curves) which then implies that $a_{i}=-C'(\omega^{0})|_{(\partial D)_{i}}$, then the first variation $\delta H_{c}(\bv^{0})\delta \bv =0$.
\end{proof}
We denote $\|\nabla\bv\|_{2}^{2}:= \int_{D}|\nabla\bv(\bx)|^{2}d\bx$, where $|\nabla \bv|^{2}=tr(\nabla \bv \cdot (\nabla \bv)^{T})$. We are now ready to prove Arnold's first stability theorem for $\alpha$-Euler. 
\begin{theorem}\label{ArnoldIalphaeuler}
Let $\bv^{0}$ be a steady state solution of the $\alpha$-Euler equations \eqref{EE_alpha} on the multi connected domain $D$, satisfying Assumption \ref{steadystassumptionalpha}.
Suppose that
\begin{equation}\label{c1c2euleralpha}
0< \inf\limits_{(x,y)\in D}-F'(\omega^{0}(x,y)) \leq \sup\limits_{(x,y)\in D}-F'(\omega^{0}(x,y)) < + \infty.
\end{equation}
Then there exists a constant $K>0$, such that if $\bv(\cdot,t)=\bv^{0}+\delta \bv(\cdot,t)$, $t \in I$ solves the $\alpha$-Euler equations \eqref{EE_alpha} on $D$ then one has the following estimate for all times $t \in I$,
\begin{align*}\label{esteuleralpha}
&||\bv(\cdot,t)-\bv^{0}||^{2}_{2} + \alpha^{2}||\nabla(\bv(\cdot,t)-\bv^{0})||^{2}_{2} + ||\omega(\cdot,t)-\omega^{0}||^{2}_{2} \\& \leq K (||\bv(\cdot,0)-\bv^{0}||^{2}_{2} + \alpha^{2}||\nabla(\bv(\cdot,0)-\bv^{0})||^{2}_{2}+ ||\omega(\cdot,0)-\omega^{0}||^{2}_{2}), \numberthis
\end{align*}
where $\omega^{0}=\curl (1-\alpha^{2}\Delta)\bv^{0}$ and $\omega(\cdot,t)=\curl (1-\alpha^{2}\Delta)\bv(\cdot,t)$. 
\end{theorem}
\begin{proof}
Let $K_{1}:= \inf_{(x,y) \in D} (-F'(\omega^{0}(x,y)))$ and $K_{2}:= \sup_{(x,y) \in D} (-F'(\omega^{0}(x,y)))$. 
Let $H_{c}$ be defined as in \eqref{hc}, and choose $C$ and $a_{i}$ as in Lemma \ref{critical pt conditions}. Since the range of $\omega^{0}$ is a connected set, \eqref{c1c2euleralpha} is equivalent to,
\begin{equation}\label{K1K2}
0<K_{1} \leq -F'(\xi) \leq K_{2}<+\infty,
\end{equation} for $\xi \in [\min\limits_{(x,y) \in \overline{D}} \omega^{0}(x,y),\max\limits_{(x,y) \in \overline{D}} \omega^{0}(x,y)]$. Using \eqref{c1c2euleralpha}, we may extend $C$ from the range of $\omega^{0}$ to $\mathbb R$ such that,
\begin{equation}\label{1C''}
K_{1} \leq C''(z) \leq K_{2},
\end{equation}
holds for every $z \in \mathbb R$. Indeed, we first extend $F$ linearly outside \[[\min\limits_{(x,y) \in \overline{D}} \omega^{0}(x,y),\max\limits_{(x,y) \in \overline{D}} \omega^{0}(x,y)]\]  to all of $\mathbb R$. We then choose $C$ such that $C'(\xi)=-F(\xi)$ for all $\xi \in \mathbb R$.
By Lemma \ref{H_inv}, $H_{c}$ is an invariant of motion and by Lemma \ref{critical pt conditions}, $\bv^{0}$ is a critical point of $H_{c}$, i.e., $\delta H_{c}(\bv^{0}) \delta \bv=0$.
Equation \eqref{hc} gives
\begin{align*}
H_{c}(\bv)-H_{c}(\bv^{0})&= \frac{1}{2}\int_{D} \bu(\bx) \cdot \bv(\bx) d \bx - \frac{1}{2}\int_{D} \bu^{0}(\bx) \cdot \bv^{0}(\bx) d \bx\\&+\int_{D} (C(\omega(\bx))-C(\omega^{0}(\bx))) d \bx+\sum_{i=0}^{n} a_{i}\int_{(\partial D)_{i}} (\bu (\mathbf s) -\bu^{0}(\mathbf s)) \cdot d \mathbf s \\
&= \frac{1}{2}\int_{D} (\bu(\bx) - \bu^{0}(\bx)) \cdot (\bv(\bx)- \bv^{0}(\bx)) d \bx \\&+ \frac{1}{2}\int_{D} \big( \bv^{0}(\bx) \cdot (\bu(\bx) - \bu^{0}(\bx)) + \bu^{0}(\bx) \cdot (\bv(\bx)-\bv^{0}(\bx)) \big) d \bx \\&+\int_{D} C'(\omega^{0}(\bx))(\omega(\bx)-\omega^{0}(\bx)) d \bx + \frac{1}{2} \int_{D} C''(\xi)(\omega(\bx)-\omega^{0}(\bx))^{2} d \bx \\&+\sum_{i=0}^{n} a_{i}\int_{(\partial D)_{i}} (\bu(\mathbf s) -\bu^{0}(\mathbf s)) \cdot d \mathbf s,
\end{align*}
where we have Taylor expanded $C(\omega(\bx))-C(\omega^{0}(\bx))$ and $\xi$ depends on $\omega^{0}(\bx)$ and $\omega(\bx)$. By virtue of the fact that the first variation is zero at $\bv^{0}$, using\eqref{H_first_var_new}, with $\delta \bv = \bv -\bv^{0}$ and $\delta \bu = \bu -\bu^{0}$  we have that, 
\begin{align*}
&\frac{1}{2}\int_{D} \bigg(\bv^{0}(\bx) \cdot (\bu(\bx) - \bu^{0}(\bx)) + \bu^{0}(\bx) \cdot (\bv(\bx)-\bv^{0}(\bx))\bigg) d \bx \\&+\int_{D} C'(\omega^{0}(\bx))(\omega(\bx)-\omega^{0}(\bx)) d \bx + \sum_{i=0}^{n} a_{i}\int_{(\partial D)_{i}} (\bu(\mathbf s) -\bu^{0}(\mathbf s)) \cdot d \mathbf s=0,
\end{align*}
whence, 
\begin{align}\label{Hc}
H_{c}(\bv)-H_{c}(\bv^{0})&=\frac{1}{2}\int_{D} (\bu(\bx) - \bu^{0}(\bx)) \cdot (\bv(\bx)- \bv^{0}(\bx)) d \bx \\ \nonumber &+ \frac{1}{2} \int_{D} C''(\xi)(\omega(\bx)-\omega^{0}(\bx))^{2} d \bx. 
\end{align}
Using the fact that $\bu-\bu^{0} = (1-\alpha^{2}\Delta)  (\bv - \bv^{0})$ and integrating the second term by parts, we see, using  \eqref{int_by_parts_alph_Eul}, that 
$\int_{D} \Delta (\bv(\bx)-\bv^{0}(\bx)) \cdot ( \bv(\bx)-\bv^{0}(\bx))d \bx = - \int_{D} |\nabla  (\bv(\bx) -\bv^{0}(\bx))|^{2} d \bx$.
Thus, 
\begin{align}\label{Pos_def_form}
& H_{c}(\bv)-H_{c}(\bv^{0})= \frac{1}{2}\int_{D}   (\bv(\bx)-\bv^{0}(\bx)) \cdot (\bv(\bx)-\bv^{0}(\bx)) d \bx  \\
&+ \frac{1}{2}\int_{D} \alpha^{2} | \nabla  (\bv(\bx)-\bv^{0}(\bx)) |^{2} d \bx +\frac{1}{2}\int_{D} C''(\xi)(\omega(\bx)-\omega^{0}(\bx))^{2} d \bx. \nonumber
\end{align}
Thus, using \eqref{1C''}
\begin{align*} 
&\frac{1}{2} \int_{D} | \bv(\bx)-\bv^{0}(\bx)|^{2}d \bx + \frac{\alpha^{2}}{2} \int_{D} | \nabla (\bv(\bx)-\bv^{0}(\bx))|^{2}d \bx  \\& +\frac{K_{1}}{2} \int_{D}(\omega(\bx)-\omega^{0}(\bx))^{2} d \bx  \leq H_{c}(\bv)-H_{c}(\bv^{0})  \leq \frac{1}{2} \int_{D} | \bv(\bx)-\bv^{0}(\bx)|^{2} d \bx \\&+ \frac{\alpha^{2}}{2} \int_{D} | \nabla (\bv(\bx)-\bv^{0}(\bx))|^{2} d \bx + \frac{K_{2}}{2} \int_{D}(\omega(\bx)-\omega^{0}(\bx))^{2} d \bx.
\end{align*} 
We now let $\beta_{1}= \min(\frac{1}{2},\frac{K_{1}}{2})$ and $\beta_{2}= \max(\frac{1}{2},\frac{K_{2}}{2})$ and see that, 
\begin{align*}
&\beta_{1}(\| \bv -\bv^{0}\|_{2}^{2} + \alpha^{2}\| \nabla (\bv -\bv^{0})\|_{2}^{2}+\| \omega -\omega^{0}\|_{2}^{2}) \leq H_{c}(\bv)-H_{c}(\bv^{0})\\ 
&\leq  \beta_{2}(\| \bv -\bv^{0}\|_{2}^{2} + \alpha^{2}\| \nabla (\bv -\bv^{0})\|_{2}^{2} + \| \omega -\omega^{0}\|_{2}^{2}).\numberthis \label{Hamiltonian_bound}
\end{align*} 
We now use Lemma \ref{H_inv}, (the fact that the Hamiltonian is a temporal invariant of the motion) to get, for any time $t \in I$, 
\begin{align*} 
&(\| \bv(t) -\bv^{0}\|_{2}^{2} + \alpha^{2}\| \nabla (\bv(t) -\bv^{0})\|_{2}^{2}+\| \omega(t) -\omega^{0}\|_{2}^{2}) \\&\leq \beta_{1}^{-1}(H_{c}(\bv(t))-H_{c}(\bv^{0}))
 =\beta_{1}^{-1}(H_{c}(\bv(0))-H_{c}(\bv^{0})) \\&\leq \beta_{2}\beta_{1}^{-1} (\| \bv(0) -\bv^{0}\|_{2}^{2} + \alpha^{2}\| \nabla (\bv(0) -\bv^{0})\|_{2}^{2}+\| \omega(0) -\omega^{0}\|_{2}^{2}).
\end{align*}
From this, \eqref{esteuleralpha} follows by putting $K=\beta_{2}\beta_{1}^{-1}$. %
\end{proof}
We will now address Arnold's second theorem for $\alpha$-Euler. 
We compute the second variation of $H_c$, where $\delta \bu =(1-\alpha^{2}\Delta)\delta \bv $ and $\delta \omega = \curl (1-\alpha^{2}\Delta)\delta \bv $,
\begin{align*}
&\delta^{2} H_{c}(\bv^{0})( \delta \bv, \delta \bv)  := \frac{d^{2}}{d \varepsilon^{2}}H(\bv^{0}+\varepsilon \delta \bv)\big|_{\varepsilon=0} \\
&=\frac{1}{2}\int_{D} \delta \bv \cdot \delta \bu d \bx+\frac{1}{2}\int_{D} \delta \bu \cdot \delta \bv d \bx + \int_{D}  C''(\omega^{0}) \delta \omega \delta \omega d \bx \\
&= \int_{D} (\delta \bv) \cdot \delta \bv d \bx + \alpha^{2} \int_{D} |\nabla \delta \bv|^{2}  d \bx +  \int_{D}  C''(\omega^{0}) \delta \omega \delta \omega d \bx, \numberthis \label{H_second_var}
\end{align*}
where, we integrate by parts using \eqref{int_by_parts_alph_Eul}.
\begin{remark}\label{perturb_alpha}
We note that the second variation defines the following quadratic form $K(\bv^{0})$ on the space $X$ defined in \eqref{x}.
\begin{align}\label{qf}
&K(\bv^{0})(\bv,\bv) := \delta^{2}H_{c}(\bv^{0})(\bv,\bv)= \int_{D} \bv \cdot \bv d \bx \\ \nonumber
&+ \alpha^{2} \int_{D} |\nabla \bv|^{2}  d \bx + \int_{D}C''(\omega^{0})(\curl(1-\alpha^{2}\Delta) \bv)^{2}  d \bx. 
\end{align}
Under assumption \eqref{c1c2euleralpha}, the second variation defined by \eqref{qf} is bounded and positive definite on the space $X$. Note that if there exists a $\bv \neq 0 \in X$ such that $\curl (1-\alpha^{2}\Delta) \bv=0$, then the value of the quadratic form reduces to $\int_{D} \bv \cdot \bv d \bx+ \alpha^{2} \int_{D} \nabla \bv \cdot \nabla \bv d \bx$ and this cannot be negative definite. In the proof of Arnold's second theorem given below, we will require the quadratic form defined by \eqref{qf} to be negative definite. We would like to restrict the perturbations $\delta \bv$ to a subspace of $X$ such that the operator $\curl (1-\alpha^{2}\Delta)$ is one to one and thus the quadratic form \eqref{qf} can be negative definite under appropriate assumptions on $C''$. We thus restrict the perturbation stream function to the following subspace,
\begin{align*}
Y_{\alpha}=\bigg\{&\phi:H^{4}(D;\mathbb R): \phi|_{(\partial D)_{0}}=0;  \int_{(\partial D)_{i}} -\nabla^{\perp}(1-\alpha^{2}\Delta) \phi \cdot d \mathbf{s} = 0, 1 \leq i \leq n ;  \\&(\bn \cdot \nabla)(\nabla^{\perp}\phi) \parallel \bn \mbox{ on } \partial D; \phi|_{(\partial D)_{i}} \mbox{ is constant}, 1 \leq i \leq n \bigg\}. \numberthis \label{st}
\end{align*}
We note that we do not specify the exact values of the constant that $\phi$ takes along the inner boundary curves. 
Also, choose for the velocity perturbations the subspace of $X$ given by 
\begin{align*}
X_{\alpha}:=\bigg\{\bv\in H^{3}(D;\mathbb R^{2}), &\div \bv =0 \mbox{ in } D, \int_{(\partial D)_{i}}(1-\alpha^{2}\Delta)\bv \cdot d\mathbf{s}=0,1 \leq i\leq n,\\& \bv \cdot \bn =0 \mbox{ on } \partial D, (\bn \cdot \nabla) \bv \parallel \bn \mbox{ on } \partial D\bigg\}.
\end{align*} 
\begin{lemma}\label{gradperpontoalpha}
The operator $-\nabla^{\perp}:Y_{\alpha} \to X_{\alpha}$ is bijective. That is, given $\bv \in X_{\alpha}$, there exists a unique $\phi \in Y_{\alpha}$ such that $\bv=-\nabla^{\perp}\phi$.
\end{lemma}
The proof of this Lemma is omitted as it follows from standard arguments in vector calculus, see for example \cite[pp. 166-168]{N10}. The proof does not use the conditions on the circulations in the definitions of both spaces $X_{\alpha}$ and $Y_{\alpha}$. We will need those in the proof of Lemma \ref{lapuniqalpha}.
\begin{lemma}\label{lapuniqalpha}
The operator $-\Delta(1-\alpha^{2}\Delta):Y_{\alpha}\to L_{2}$ is one to one. That is, if $-\Delta(1-\alpha^{2}\Delta) \phi =0$, for some $\phi \in Y_{\alpha}$, then, $\phi=0$.
\end{lemma}
\begin{proof}
Note that $\phi$ satsfies,
\begin{align} %
&-\Delta(1-\alpha^{2}\Delta) \phi=0, \label{circfixuniq1alpha}\\
&\phi(x,y)|_{(\partial D)_0}=0, \label{circfixuniq2alpha}\\
&\phi|_{(\partial D)_{i}}(x,y)=c_{i}, \mbox{ for } 1 \leq i \leq n, \numberthis \label{circfixuniq3alpha}\\ 
&\int_{(\partial D)_{i}} -\nabla^{\perp}(1-\alpha^{2}\Delta) \phi(s) \cdot d \mathbf{s} = 0 \mbox{ for } 1 \leq i \leq n, \label{circfixuniq4alpha}\\
&(\bn \cdot \nabla)  \nabla^{\perp}\phi \cdot \bt =0. \label{circfixuniq5alpha}
\end{align}
Multiply  \eqref{circfixuniq1alpha} by $\phi$ and integrate over the domain to get
\begin{align*}
0&=\int_{D} \phi (-\Delta(1-\alpha^{2}\Delta) \phi) d \bx = -\sum_{i=0}^{n} \int_{(\partial D)_{i}}\phi \bn \cdot (1-\alpha^{2}\Delta) \nabla \phi ds \\&+ \int_{D} \nabla \phi \cdot (1-\alpha^{2}\Delta) \nabla \phi d \bx \\&=  \sum_{i=0}^{n} \phi|_{(\partial D)_{i}} \int_{(\partial D)_{i}}  (1-\alpha^{2}\Delta)\nabla^{\perp} \phi \cdot d \mathbf{s}  +\int_{D} \nabla \phi \cdot (1-\alpha^{2}\Delta) \nabla \phi d \bx \\&= \int_{D} \nabla \phi \cdot (1-\alpha^{2}\Delta) \nabla \phi d \bx, \numberthis \label{greensformula}
\end{align*}
where we have used Green's formula and \eqref{circfixuniq2alpha}, \eqref{circfixuniq3alpha} and \eqref{circfixuniq4alpha} and the fact that $\bn \cdot \nabla \phi = - \bt \cdot \nabla^{\perp} \phi$. 
But, 
\[\int_{D} \nabla \phi \cdot (1-\alpha^{2}\Delta) \nabla \phi  d \mathbf{x}= \int_{D} \nabla \phi \cdot \nabla \phi  d \mathbf{x} -\alpha^{2}\int_{D} \nabla \phi \cdot \Delta \nabla \phi  d \mathbf{x}.\] 
By \eqref{int_by_parts_alph_Eul}, we have that \[ \int_{D} \nabla \phi \cdot \Delta \nabla \phi  d \mathbf{x} =\int_{D} -\nabla^{\perp} \phi \cdot \Delta (-\nabla^{\perp} \phi)  d \mathbf{x} = \int_{D} \bv \cdot \Delta \bv d \mathbf{x} = -\int_{D} |\nabla \bv|^{2} d \mathbf{x}, \] 
where $\bv \in X_{\alpha}$ is the unique solution to $\bv=-\nabla^{\perp}\phi$, via Lemma \ref{gradperpontoalpha}.
Thus,  
\begin{align*} 
&\int_{D} \nabla \phi \cdot (1-\alpha^{2}\Delta) \nabla \phi  d \mathbf{x}= \int_{D} \nabla \phi \cdot \nabla \phi  d \mathbf{x} + \alpha^{2} \int_{D} |\nabla \bv|^{2}  d \mathbf{x} \\&=\int_{D} \bv \cdot \bv  d \mathbf{x} + \alpha^{2} \int_{D} |\nabla \bv|^{2}  d \mathbf{x} = 0,
\end{align*} from which we conclude that $\bv=0$.
It follows by Lemma \ref{gradperpontoalpha} that $-\nabla^{\perp} \phi =0$ on $D$ and hence $\nabla \phi=0$. Then $\phi$ is a constant, and is equal to $0$ by \eqref{circfixuniq2alpha}.
\end{proof}
\begin{lemma}\label{curl11alpha}
The operator $\curl(1-\alpha^{2}\Delta): X_{\alpha} \to L_{2}$ is one to one. That is, if $\curl(1-\alpha^{2}\Delta) \bv =0$, for some $\bv \in X_{\alpha}$, then $\bv =0$. 
\end{lemma}
\begin{proof}
Let $\bv \in X_{\alpha}$ be such that $\curl(1-\alpha^{2}\Delta) \bv =0$. By Lemma \ref{gradperpontoalpha}, we have that there exists a $\phi \in Y_{\alpha}$ such that $\bv=-\nabla^{\perp}\phi$. Then $\curl(1-\alpha^{2}\Delta) \bv =\curl(1-\alpha^{2}\Delta) (-\nabla^{\perp}\phi) = -\Delta (1-\alpha^{2}\Delta) \phi =0$. By Lemma \ref{lapuniqalpha}, we have that $\phi=0$. Thus $\bv=-\nabla^{\perp}\phi=0$.
\end{proof}
\end{remark}
\begin{remark}\label{rrformalpha}
We comment on the Rayleigh-Ritz formula frequently used in hydrodynamics, see, e.g. (\cite[Lemma 4.16, page 111]{S00}). Let $A$ be a positive operator with compact resolvent acting in a Hilbert space $\mathcal{H}$. The minimal eigenvalue $\lambda_{min}(A)$ can be computed by the following Rayleigh-Ritz formula:
\begin{equation}\label{rrf1}
\lambda_{min}(A)=\min_{0 \neq \phi \in \mbox{dom } A} \frac{||A\phi||^{2}}{\langle A\phi,\phi \rangle},
\end{equation}
where $||\cdot||$ and $\langle \cdot , \cdot \rangle$ are the norm and the scalar product in $\mathcal{H}$. Indeed, as the following argument shows, \eqref{rrf1} is a consequence of the standard (Hilbert-Schmidt-Courant-Fischer) minimax principle: Let $B=A^{-1}$. By assumptions, $B$ is a compact positive operator whose maximal eigenvalue is given by the formula (see, e.g. \cite[Sec. XIII.1, page 76 onwards]{RS78})
\begin{equation}\label{rrf2}
\lambda_{max}(B)=\max_{\psi \neq 0} \frac{\langle B\psi,\psi\rangle}{||\psi||^{2}}.
\end{equation}
Letting $\psi=A\phi$ and using the spectral mapping theorem $sp(A)=(sp(B))^{-1}$,
\begin{equation*}
\lambda_{min}(A)=(\lambda_{max}(B))^{-1} = \bigg(\max_{0 \neq \phi \in \mbox{dom }A} \frac{\langle \phi, A \phi\rangle}{||A\phi||^{2}}\bigg)^{-1}=\min_{0 \neq \phi \in \mbox{dom} A} \frac{||A\phi||^{2}}{\langle A\phi,\phi \rangle},
\end{equation*}
yielding \eqref{rrf1}.
We apply formula \eqref{rrf1} for the following situation. Let $A=-\Delta(1-\alpha^{2}\Delta)$ with the domain $\mbox{dom }A=Y_{\alpha} \subset L^{2}(D)$, see \eqref{st} for the definition of $Y_{\alpha}$. We note that $Y_{\alpha} \subset L^{2}(D)$ is compactly embedded in $L^{2}(D)$ by the standard Sobolev embedding. Due to the choice of the boundary conditions, 
\begin{equation}\label{rrf5}
\langle A\phi, \phi\rangle_{L^{2}}=\langle \phi, -\Delta(1-\alpha^{2}\Delta)\phi\rangle_{L^{2}}=||\bv||_{2}^{2}+\alpha^{2}||\nabla\bv||_{2}^{2}
\end{equation}
for all $\phi \in Y_{\alpha}$ and $\bv=-\nabla^{\perp}\phi \in X_{\alpha}$. In particular, the operator $A$ is positive. On the other hand, the formula,
\begin{equation}
\curl(1-\alpha^{2}\Delta)\bv=-\Delta(1-\alpha^{2}\Delta)\phi
\end{equation}
yields
\begin{equation}\label{rrf6}
||A\phi||_{2}^{2}=||-\Delta(1-\alpha^{2}\Delta)\phi||_{2}^{2}=||\curl (1-\alpha^{2}\Delta)\bv||_{2}^{2}.
\end{equation}
Combining \eqref{rrf1}, \eqref{rrf5}, \eqref{rrf6}, we obtain the following analogue of the Rayleigh-Ritz formula for the $\alpha$-Euler equation:
\begin{equation}\label{lminalpha}
\lambda_{min,\alpha}=\min_{0 \neq \bv \in X_{\alpha}} \frac{||\curl (1-\alpha^{2}\Delta)\bv||_{2}^{2}}{||\bv||_{2}^{2}+\alpha^{2}||\nabla\bv||_{2}^{2}},
\end{equation}
where $\lambda_{min,\alpha}$ is the minimum eigenvalue of the operator $A=-\Delta(1-\alpha^{2}\Delta)$ with the domain $\mbox{dom }A=Y_{\alpha} \subset L^{2}(D)$.
\end{remark}
By equation \eqref{lminalpha}, for every $\bv \in X_{\alpha}$ we have,
\begin{equation}\label{lambdaminalphaineq}
\lambda_{min,\alpha} \bigg(\int_{D} \bv \cdot \bv  d \mathbf{x} + \alpha^{2} \int_{D} \nabla \bv \cdot \nabla \bv d \mathbf{x}\bigg) \leq \int_{D} (\curl(1-\alpha^{2}\Delta)\bv)^{2}d \bx. 
\end{equation}
We shall now prove Arnold's second stability theorem for $\alpha$-Euler.
\begin{theorem}\label{Arnold II alpha}
Let $\bv^{0}$ be a steady state solution of the $\alpha$-Euler equations \eqref{EE_alpha} on the multi connected domain $D$, satisfying Assumption \ref{steadystassumptionalpha}. Let $\lambda_{min, \alpha}>0$ be the minimum eigenvalue of the operator $-\Delta(1-\alpha^{2}\Delta):L^{2}(D) \to L^{2}(D)$ with the domain $\mbox{dom}(-\Delta(1-\alpha^{2}\Delta))=Y_{\alpha}$.
Suppose
\begin{equation}\label{c1c2euleralphaneg}
0< \frac{1}{\lambda_{min,\alpha}}< \inf\limits_{(x,y)\in D}F'(\omega^{0}(x,y)) \leq \sup\limits_{(x,y)\in D}F'(\omega^{0}(x,y)) < + \infty.
\end{equation}
There exists a constant $K>0$, such that if $\bv(\cdot,t)=\bv^{0}+\delta \bv(\cdot,t)$, $t \in I$ solves the $\alpha$-Euler equations \eqref{EE_alpha} on $D$, with $\delta \bv \in X_{\alpha}$, then one has the following estimate for all times $t \in I$,
\begin{align*}\label{esteuleralpha}
&||\bv(\cdot,t)-\bv^{0}||^{2}_{2} + \alpha^{2}||\nabla(\bv(\cdot,t)-\bv^{0})||^{2}_{2} + ||\omega(\cdot,t)-\omega^{0}||^{2}_{2} \\& \leq K (||\bv(\cdot,0)-\bv^{0}||^{2}_{2} + \alpha^{2}||\nabla(\bv(\cdot,0)-\bv^{0})||^{2}_{2}+ ||\omega(\cdot,0)-\omega^{0}||^{2}_{2}), \numberthis
\end{align*}
where $\omega^{0}=\curl (1-\alpha^{2}\Delta)\bv^{0}$ and $\omega(\cdot,t)=\curl (1-\alpha^{2}\Delta)\bv(\cdot,t)$.
\end{theorem}
\begin{proof}
Let $K_{1}:= \inf_{(x,y) \in D} F'(\omega^{0}(x,y))$ and $K_{2}:= \sup_{(x,y) \in D} F'(\omega^{0}(x,y))$. From the fact that $C''(\omega^{0}(x,y))=-F'(\omega^{0}(x,y))$ for every $(x,y) \in D$, we see that 
\[ 0<K_{1} \leq -C''(\omega^{0}(x,y)) \leq K_{2}<+\infty,\]
for every $(x,y) \in D$. We first extend $C$ to all of $\mathbb R$ such that 
\begin{equation}\label{C''}
K_{1} \leq -C''(\xi) \leq K_{2},
\end{equation}
holds for every $\xi \in \mathbb R$.
Proceeding similarly to the proof of Arnold's first theorem, we obtain that, cf. \eqref{Pos_def_form}, 
\begin{align}\label{Pos_def_form2}
& H_{c}(\bv)-H_{c}(\bv^{0})= \frac{1}{2}\int_{D}   (\bv-\bv^{0}) \cdot (\bv-\bv^{0}) d \bx  \\
&+ \frac{1}{2}\int_{D} \alpha^{2}  |\nabla  (\bv-\bv^{0})|^{2} d \bx +\frac{1}{2}\int_{D} C''(\xi)(\omega-\omega^{0})^{2} d \bx. \nonumber
\end{align}
Since \eqref{C''} holds, we have that, 
\begin{align*}
& -\frac{1}{2}\int_{D} |\bv-\bv^{0}|^{2}d \bx - \frac{\alpha^{2}}{2}\int_{D} |\nabla(\bv-\bv^{0})|^{2}d \bx+\frac{K_{1}}{2}\int_{D}(\omega-\omega^{0})^{2}d \bx \\&\leq H_{c}(\bv^{0})-H_{c}(\bv) \leq - \frac{1}{2}\int_{D} |\bv-\bv^{0}|^{2}d \bx - \frac{\alpha^{2}}{2}\int_{D} |\nabla(\bv-\bv^{0})|^{2}d \bx \\&+\frac{K_{2}}{2}\int_{D}(\omega-\omega^{0})^{2}d \bx. \numberthis \label{Hvineq}
\end{align*}
By \eqref{lambdaminalphaineq}, we have that,
\begin{align*}%
\int_{D} (\bv-\bv^{0}) \cdot (\bv-\bv^{0}) d \bx + \alpha^{2} \int_{D} |\nabla (\bv-\bv^{0})|^{2} d \bx  \leq \frac{1}{\lambda_{min,\alpha}} \int_{D}  (\omega -  \omega^{0}))^{2}  d \bx, \numberthis \label{Poincare_like_ineq_2}
\end{align*} 
This then means that the left hand side of \eqref{Hvineq} can be estimated from below by
\begin{align*} 
&0 < \frac{(K_{1}- 1/\lambda_{min,\alpha})}{2}\int_{D}((\omega-\omega^{0})^{2}d \bx \leq -\frac{1}{2}\int_{D} |\bv-\bv^{0}|^{2}d \bx \\&- \frac{\alpha^{2}}{2}\int_{D} |\nabla(\bv-\bv^{0})|^{2}d \bx+\frac{K_{1}}{2}\int_{D}(\omega-\omega^{0})^{2}d \bx \leq H_{c}(\bv^{0})-H_{c}(\bv).
\end{align*}
Thus, obviously, splitting the LHS, we obtain
\begin{align*}
&\frac{1}{4}(K_{1}- 1/\lambda_{min,\alpha})\int_{D}((\omega-\omega^{0})^{2}d \bx) \\&+ \frac{1}{4}(K_{1}- 1/\lambda_{min,\alpha})\int_{D}((\omega-\omega^{0})^{2}d \bx \leq H_{c}(\bv^{0})-H_{c}(\bv).
\end{align*}
Using \eqref{lambdaminalphaineq} again, we see that 
\begin{align*}
&\frac{\lambda_{min,\alpha}(K_{1}-1/\lambda_{min,\alpha})}{4} \bigg(\int_{D}|\bv-\bv^{0}|^{2}d \bx + \alpha^{2}\int_{D} |\nabla(\bv-\bv^{0})|^{2}d \bx \bigg)\\& + \frac{1}{4}(K_{1}- 1/\lambda_{min,\alpha})\int_{D}((\omega-\omega^{0})d \bx) \leq H_{c}(\bv^{0})-H_{c}(\bv). \numberthis \label{beta1}
\end{align*}
On the other hand the right hand side of \eqref{Hvineq} can be estimated as follows:
\begin{align*}
&H_{c}(\bv^{0})-H_{c}(\bv) \leq \frac{-1}{2}\int_{D} |\bv-\bv^{0}|^{2}d \bx - \frac{\alpha^{2}}{2}\int_{D} |\nabla(\bv-\bv^{0})|^{2}d \bx \\&+K_{2}\int_{D}(\omega-\omega^{0})d \bx \leq \frac{1}{2}\int_{D} |\bv-\bv^{0}|^{2}d \bx \\&+ \frac{\alpha^{2}}{2}\int_{D} |\nabla(\bv-\bv^{0})|^{2}d \bx +K_{2}\int_{D}(\omega-\omega^{0})d \bx. \numberthis \label{beta2}
\end{align*}
Now let $\beta_{1}=\min \{\frac{\lambda_{min,\alpha}(K_{1}-1/\lambda_{min,\alpha})}{4}, \frac{1}{4}(K_{1}- 1/\lambda_{min,\alpha})\}$ and $\beta_{2}= \max \{ \frac{1}{2},K_{2}\}$ and we obtain, cf. \eqref{Hamiltonian_bound},
\begin{align*}
&\beta_{1}(\| \bv -\bv^{0}\|_{2}^{2} + \alpha^{2}\| \nabla (\bv -\bv^{0})\|_{2}^{2}+\| \omega -\omega^{0}\|_{2}^{2}) \leq H_{c}(\bv^{0})-H_{c}(\bv)\\ 
&\leq  \beta_{2}(\| \bv -\bv^{0}\|_{2}^{2} + \alpha^{2}\| \nabla (\bv -\bv^{0})\|_{2}^{2} + \| \omega -\omega^{0}\|_{2}^{2}).
\end{align*}
We can thus finish the proof as in Arnold's first theorem.
\end{proof}
\subsection{Arnold's theorems in a bounded, simply connected domain}
Let $D \subset \mathbb R^{2}$  be a bounded, simply connected region with a smooth boundary $\partial D$. The functional \eqref{hc} is now given by,
\beq\label{hcalphaeulersimplyconnected}
H_{c}(\bv)= \frac{1}{2}\int_{D} \bv \cdot (1-\alpha^{2}\Delta)\bv d \bx + \int_{D} C(\curl(1-\alpha^{2}\Delta) \bv) d \bx +  a \int_{\partial D}(1-\alpha^{2}\Delta)\bv \cdot \mathbf{ds}. 
\enq 
Lemma \ref{H_inv} holds and we impose Assumption \ref{steadystassumptionalpha}. Lemma \ref{critical pt conditions} also holds, where we now set $a= F(\omega^{0})|_{\partial D}$. The expression for the second variation given in \eqref{H_second_var} remains unchanged. Theorem \ref{ArnoldIalphaeuler} also holds in this case. We now expand upon Remark \ref{perturb_alpha}.
\begin{remark}\label{perturbsimplyconnectedalpha}
Our Lemmas \ref{gradperpontoalpha}, \ref{lapuniqalpha} and \ref{curl11alpha} in Remark \ref{perturb_alpha} will work where the subspace for the stream function perturbations is now $Y_{\alpha}:=\{\phi \in H^{4}(D;\mathbb R) \cap H^{1}_{0}(D;\mathbb R); (\bn \cdot \nabla) (\nabla \phi)\cdot \bn = 0 \mbox{ on } \partial D\}$ and the subspace for the velocity perturbations is 
\begin{align*}
X_{\alpha}:=\bigg\{\bv\in H^{3}(D;\mathbb R^{2}), \div \bv =0 \mbox{ in } D,  \bv \cdot \bn =0 \mbox{ and }  (\bn \cdot \nabla) \bv \parallel  \bn \mbox{ on } \partial D \bigg\}.
\end{align*}
The proofs are similar and are omitted. Arnold's second theorem then follows as stated. 
\end{remark}
\subsection{Arnold's second theorem on the two torus}\label{twotorus}
By Remark \ref{boundaryalpha} below, we do not expect that Arnold's first theorem holds on the two torus.
\begin{remark}\label{boundaryalpha}
\sloppy
Condition \eqref{c1c2euleralpha} in Arnold's first theorem \ref{ArnoldIalphaeuler} is never satisfied in a domain without a boundary , see \cite[Section 3.2, page 112]{MP94}. 
To demonstrate this, let us assume \eqref{c1c2euleralpha}. Then, since $F$ is monotone, there exists its inverse function denoted by $G$, i.e., $G=F^{-1}$, and since \eqref{c1c2euleralpha} holds, %
one has the relationship 
\begin{equation}\label{gineqalpha}
\frac{1}{c_{2}} \leq - G'(\xi) \leq \frac{1}{c_{1}},
\end{equation}
 for all $\xi$ in the range of $F(\omega^{0}(\cdot,\cdot))$, i.e., for all $\xi$ in the range of $\phi^{0}(\cdot,\cdot)$. In particular, $G'$ is negative everywhere. Assume without loss of generality that $\partial_{x}\phi^{0} \neq 0$ (if it is, then in the argument below replace $\partial_{x}\phi^{0}$ by $\partial_{y}\phi^{0}$, we exclude the trivial case $\phi^{0}=$ constant everywhere in $D$). We have that $-\Delta (1-\alpha^{2}\Delta)\phi^{0}=\omega^{0}=G(\phi^{0})$. From this we see that, $\partial_{x}(-\Delta(1-\alpha^{2}\Delta) \phi^{0})=G'(\phi^{0})\partial_{x}\phi^{0}$. Multiplying this by $\partial_{x}\phi^{0}$ and integrating this over the domain $D$, we get, 
\begin{equation}\label{g}
 -\int_{D} \partial_{x}\phi^{0} \partial_{x} \Delta (1-\alpha^{2}\Delta) \phi^{0} d \bx = \int_{D} G'(\phi^{0})(\partial_{x} \phi^{0})^{2} d \bx.
\end{equation}
Integrating the left hand side by parts, we get, 
\begin{align*}
&-\int_{D} \partial_{x}\phi^{0} \partial_{x} \Delta (1-\alpha^{2}\Delta) \phi^{0} d \bx = -\int_{D} \partial_{x}\phi^{0}  \Delta (1-\alpha^{2}\Delta) \partial_{x}\phi^{0} d \bx \\
&=\int_{D} \nabla(\partial_{x}\phi^{0})\cdot \nabla ((1-\alpha^{2}\Delta) \partial_{x} \phi^{0}) d \bx - \int_{\partial D} (\partial_{x}\phi^{0}) \bn \cdot \nabla (\partial_{x} (1-\alpha^{2}\Delta)\phi^{0}) ds.
\end{align*}
But 
\begin{align*}
&\int_{D} \nabla(\partial_{x}\phi^{0})\cdot \nabla ((1-\alpha^{2}\Delta) \partial_{x} \phi^{0}) d \bx = \int_{D} (\nabla \partial_{x}\phi^{0})^{2} d \bx -\\& \alpha^{2} \int_{D} \nabla (\partial_{x} \phi^{0}) \cdot \nabla (\Delta \partial_{x} \phi^{0}) d \bx = \int_{D} (\nabla \partial_{x}\phi^{0})^{2} d \bx \\&+ \alpha^{2} \int_{D} |\nabla(\nabla \partial_{x} \phi^{0})|^{2}d \bx  -\int_{\partial D} \nabla (\partial_{x}\phi^{0}) \bn \cdot \nabla (\nabla \partial_{x} \phi^{0}) ds.
\end{align*}
Thus, rewriting the left side of \eqref{g} one obtains
\begin{align*}
&\int_{D} (\nabla \partial_{x}\phi^{0})^{2} d \bx + \alpha^{2} \int_{D} |\nabla(\nabla \partial_{x} \phi^{0})|^{2}d \bx  -\int_{\partial D} \nabla (\partial_{x}\phi^{0}) \bn \cdot \nabla (\nabla \partial_{x} \phi^{0}) ds \\ & -\int_{\partial D} (\partial_{x}\phi^{0}) \bn \cdot \nabla (\partial_{x} (1-\alpha^{2}\Delta)\phi^{0}) ds =\int_{D} G'(\phi^{0})(\partial_{x} \phi^{0})^{2} d \bx .
\end{align*}
Note that the first two terms on the left hand side are positive and the term in the right hand side is negative by \eqref{gineqalpha} which leads to a contradiction in the absence of the boundary terms in the left hand side. 
\end{remark}
We consider Arnold's second theorem on the two torus $\mathbb T^{2}$. %
The Hamiltonian $H_{c}$ is now given by,
\beq\label{Hceuleralphatorus}
H_{c}(\bv)= \frac{1}{2}\int_{D} \bv \cdot (1-\alpha^{2})\bv d \bx + \int_{D} C(\curl(1-\alpha^{2}\Delta) \bv) d \bx. 
\enq 
Lemma \ref{H_inv} remains true in this setting. We also impose Assumption \ref{steadystassumptionalpha}. Lemma \ref{critical pt conditions} is modified as follows.
\begin{lemma}\label{critical pt conditions torus}
Let $\bv^{0}$, $\omega^{0}$ be a steady state solution of \eqref{EE_alpha}, satisfying Assumption \ref{steadystassumptionalpha}, where $\omega^{0}=\curl(1-\alpha^{2}\Delta)\bv^{0}$. 
 Let $C$ be a smooth function  so that 
\begin{equation}\label{Cg}
C'(\omega^{0}(x,y))=-F(\omega^{0}(x,y)),
\end{equation} for every $(x,y) \in D$.  
Then $\delta H_{c}(\bv^{0})\delta \bv=0$, i.e $\bv^{0}$ is a critical point of $H_{c}$. 
\end{lemma}
The expression for the second variation remains the same as \eqref{H_second_var}. The statement of Arnold's first theorem remains the same as in Theorem \ref{ArnoldIalphaeuler}. Remark \ref{perturb_alpha} is modified as follows.
\begin{remark}\label{perturbalphatwotorus}
Our space for the perturbation stream function is now given by 
\begin{equation}\label{yalphatorus}
Y_{\alpha}:=\{\phi \in H^{4}(\mathbb T^{2}); \int_{\mathbb T^{2}}\phi d \bx = 0\},
\end{equation}
and for the velocities is given by $X_{\alpha}:=\{\bv \in H^{3}(\mathbb T^{2}; \mathbb R^{2}); \int_{\mathbb T^{2}} \bv d \bx= 0; \div \bv =0\}$. Lemmas \ref{gradperpontoalpha}, \ref{lapuniqalpha} and \ref{curl11alpha} in Remark \ref{perturb_alpha} are true in this setting with minor modifications in the proof. 
Arnold's second theorem then follows as stated in Theorem \ref{Arnold II alpha}.
\end{remark}
\subsection{Arnold's theorems on the periodic channel}
Now, we would like to formulate Arnold's theorems on the periodic channel $D=\mathbb T \times [-1,1]$, so that the boundary conditions are periodic in the $x$ direction  with boundary conditions $\bv \cdot \bn$ and $(\bn \cdot \nabla) \bv $ parallel to $\bn$ at the ``walls'' $y=1$ and $y=-1$.
We will prove that since the domain is translationally invariant in the $x$ direction, the $x$ momentum is conserved, i.e., we will prove that, if $\bv(t,\cdot)=(v_{1}(t,\cdot),v_{2}(t,\cdot))$, $\bu(t,\cdot)=(u_{1}(t,\cdot),u_{2}(t,\cdot))$ solve the  $\alpha$-Euler equation \eqref{EE_alpha}, then 
\begin{equation}\label{momentumalpha}
M_{x}=\int_{-1}^{1}\int_{\mathbb T} u_{1}(t,x,y)dx dy = \int_{-1}^{1}\int_{\mathbb T} u_{1}(0,x,y)dx dy
\end{equation}
is an invariant of the motion. Here $\bu=(1-\alpha^{2}\Delta)\bv$.
\begin{lemma}\label{minvalpha}
Suppose $\bv(t,\cdot)=(v_{1}(t,\cdot),v_{2}(t,\cdot))$, $\bu(t,\cdot)=(u_{1}(t,\cdot),u_{2}(t,\cdot))$ solve the  $\alpha$-Euler equation \eqref{EE_alpha}, on the domain $\mathbb T \times [-1,1]$. Then 
\begin{equation}\label{minveq}
\frac{d}{dt} M_{x}=\frac{d}{dt} \int_{-1}^{1}\int_{\mathbb T} u_{1}(t,x,y)dx dy=0.
\end{equation}
\end{lemma}
\begin{proof}
We first note that the boundary conditions $\bv \cdot \bn|_{y=\pm1} =0$ imply that 
\beq\label{v2bcalpha}
v_{2}(x,-1)=v_{2}(x,1)=0.
\enq
Also, the boundary condition $\bn \cdot \nabla \bv$ parallel to $\bn$ implies that $\bn \cdot \nabla \bv \cdot \bt=0$. On the boundaries $y=-1$ and $y=1$, $\bn=(0,\pm1)$ and thus $\bn \cdot \nabla = \partial_{y}$. 

Thus $\bn \cdot \nabla \bv = \partial_{y} \bv =(\phi_{yy},-\phi_{yx})$. Thus $(\phi_{yy},-\phi_{yx}) \cdot (1,0)=0$ implies that $\phi_{yy}=0$ on the boundary, i.e., 
\begin{equation}\label{bc2alpha}
\phi_{yy}(x,-1)=\phi_{yy}(x,1)=0.
\end{equation}
The $\alpha$-Euler equations \eqref{EE_alpha} can be rewritten as, see \cite[Eq 8.33, page 67]{HMR98},
\begin{equation}\label{alphaeulernew}
\partial_{t}\bu - \bv \times (\nabla \times \bu)+ \nabla (\bv \cdot \bu - \frac{1}{2}|\bv|^{2}-\frac{\alpha^{2}}{2}|\nabla \bv|^{2}+p)=0.
\end{equation}
Denote $f=\bv \cdot \bu - \frac{1}{2}|\bv|^{2}-\frac{\alpha^{2}}{2}|\nabla \bv|^{2}+p$. Also note that $\nabla \times \bu = (\partial_{x}u_{2}-\partial_{y}u_{1})\mathbf k$, where $\mathbf k$ is the unit vector pointing out of the plane of flow. Thus $\bv \times (\nabla \times \bu)=(v_{2}\partial_{x}u_{2}-v_{2}\partial_{y}u_{1})\mathbf i - (v_{1}\partial_{x}u_{2}-v_{1}\partial_{y}u_{1})\mathbf j$.
We thus have that, using \eqref{alphaeulernew} and the computations above,
\begin{align*}
&\frac{d}{dt} M_{x}=\frac{d}{dt} \int_{-1}^{1}\int_{\mathbb T} u_{1}(t,x,y)dx dy = \int_{-1}^{1}\int_{\mathbb T} \partial_{t}u_{1}(t,x,y)dx dy\\
&=\int_{-1}^{1}\int_{\mathbb T} \big( v_{2}\partial_{x}u_{2}-v_{2}\partial_{y}u_{1}- \partial_{x}f \big) dx dy.
\end{align*}
We analyze this term by term. Notice first that $\int_{\mathbb T}- \partial_{x}f dx =0$.  
Rewriting $v_{2}\partial_{x}v_{2}=\partial_{x}(\frac{1}{2}v_{2})^{2}$ we get $\int_{\mathbb T} v_{2}\partial_{x}v_{2} dx = \int_{\mathbb T} \partial_{x}(\frac{1}{2}v_{2})^{2} dx=0$. Also,
\begin{align*}
&\int_{-1}^{1}\int_{\mathbb T} v_{2}\partial_{x}\Delta v_{2} dx dy = \int_{\mathbb T} \int_{-1}^{1} v_{2} \Delta (\partial_{x} v_{2}) dy dx \\&= -\int_{\mathbb T} \int_{-1}^{1} \nabla v_{2} \cdot \nabla (\partial_{x} v_{2}) dy dx=-\int_{-1}^{1} \int_{\mathbb T} \partial_{x} (\frac{1}{2}(\nabla v_{2})^{2}) dx dy = 0.
\end{align*}
where we integrate by parts in $y$ and boundary terms disappear by using boundary condition \eqref{v2bcalpha} and then switch order of integration and use the fact that \[\nabla v_{2} \cdot \nabla (\partial_{x} v_{2})=(\frac{1}{2}(\nabla v_{2})^{2}).\]
Thus, $\int_{-1}^{1}\int_{\mathbb T} v_{2}\partial_{x}u_{2} dx dy = 0$. We now look at the second term, 
\begin{align*}
&\int_{-1}^{1}\int_{\mathbb T} v_{2}\partial_{y}u_{1} dx dy = \int_{\mathbb T}\int_{-1}^{1} v_{2}\partial_{y}u_{1} dy dx \\&= - \int_{\mathbb T}\int_{-1}^{1} \partial_{y}v_{2} u_{1} dy dx =  \int_{\mathbb T}\int_{-1}^{1} \partial_{x}v_{1} u_{1} dy dx,
\end{align*}
where we integrate by parts and the boundary terms vanish using boundary condition \eqref{v2bcalpha} and since $\div \bv=0$, we have that $\partial_{x}v_{1}=-\partial_{y}v_{2}$. Notice that \[ \partial_{x}v_{1} u_{1}= \partial_{x}v_{1} v_{1} - \alpha^{2}\partial_{x}v_{1}\Delta v_{1}.\]
Since $\partial_{x}v_{1} v_{1}=\partial_{x}(\frac{1}{2}v_{1}^{2})$, we have that 
\begin{align*} 
&\int_{-1}^{1} \int_{\mathbb T} \partial_{x}v_{1}\Delta v_{1} dx dy = - \int_{-1}^{1} \int_{\mathbb T} \nabla \partial_{x}v_{1} \cdot \nabla v_{1} dx dy
\\&= - \int_{-1}^{1} \int_{\mathbb T} \partial_{x}(\frac{1}{2}(\nabla v_{1})^{2})dxdy=0.
\end{align*}
\end{proof}
Since $u_{1}=\psi_{y}$, 
using \eqref{momentumalpha}, we see that,
\begin{align*}\label{}
& \frac{1}{2\pi} M_{x}=\frac{1}{2\pi}\int_{-1}^{1}\int_{\mathbb T} u_{1}(t,x,y)dx dy = \frac{1}{2\pi}\int_{\mathbb T} \int_{-1}^{1} u_{1}(t,x,y)dy dx\\&=\frac{1}{2\pi}\int_{\mathbb T} \int_{-1}^{1} \psi_y(t,x,y)dy dx=\frac{1}{2\pi} \int_{\mathbb T} \psi(t,x,-1)dx-\frac{1}{2\pi} \int_{\mathbb T} \psi(t,x,1)dx.
\end{align*}
Since $\psi=\phi-\alpha^{2}\Delta \phi$, and by the boundary condition \eqref{bc2alpha}, $\phi_{yy}(x,\pm1)=0$, we have that
\[ \psi(x,\pm1)=\phi(x,\pm1)-\partial_{xx}\phi_{xx}(x,\pm1).\]
Also note that, \[\int_{\mathbb T} \phi_{xx}(x,\pm1) dx = \int_{\mathbb T} \partial_{x} \phi_{x}(x,\pm1) dx=0\] to conclude that
\begin{align*}
&\frac{1}{2\pi} M_{x} = \frac{1}{2\pi} \int_{\mathbb T} \psi(t,x,-1)dx-\frac{1}{2\pi} \int_{\mathbb T} \psi(t,x,1)dx \\&= \frac{1}{2\pi} \int_{\mathbb T} \phi(t,x,-1)dx-\frac{1}{2\pi} \int_{\mathbb T} \phi(t,x,1)dx. \numberthis \label{psiphibc}
\end{align*} 
By Lemma \ref{minvalpha}, $M_{x}/2 \pi$ is $t$-independent. Since $\phi(x,-1)$ and $\phi(x,1)$ are constants, one can simply take the difference to be $M_{x}/ 2 \pi$. Thus in solving the Poisson equation for the stream function we can set $\phi(x,-1)=0$ and $\phi(x,1)=M_{x}/2 \pi$.
Thus the following Poisson problem is solved to recover the stream function from the vorticity
\begin{align*}\label{dirichletEuleralpha}
-\Delta (1-\alpha^{2}\Delta) \phi=\omega, \mbox{ in } D, \quad
\phi(x,-1)=0,  \phi(x,1)=-M_{x}/ 2 \pi. \numberthis
\end{align*}
Since this must hold for both the steady state $\phi^{0}$ and the perturbed flow $\phi^{0}+\delta \phi$, we see that, the Poisson equation satisfied by the perturbation stream function $\delta \phi$ satisfies Dirichlet boundary conditions, 
\begin{align*}\label{dirichletEulerperturbalpha}
-\Delta (1-\alpha^{2}\Delta) \delta \phi=\omega, \mbox{ in } D, \quad
\delta \phi(x,-1)=0, \delta \phi(x,1)=0. \numberthis
\end{align*}
The subspace for the perturbation stream function is now as follows: 
\begin{align*}
Y_{\alpha}=\bigg\{\phi:H^{4}((\mathbb T \times [-1,1]);\mathbb R): &\phi(x,1)=0, \phi(x,-1)=0,\\&  \phi_{yy}(x,1)=0, \phi_{yy}(x,-1)=0\bigg\}.
\end{align*}
One then defines the subspace for the perturbations of velocity as 
\begin{align*}
X_{\alpha}:= & \bigg\{ \bv:H^{3}((\mathbb T \times [-1,1]);\mathbb R^{2}); \div \bv =0, \bv \cdot \bn =0 \mbox{ and } \bn \cdot \nabla \bv \parallel \bn \mbox{ on } y=\pm 1 ;\\&   \int_{-1}^{1}\int_{\mathbb T} (1-\alpha^{2}\Delta)v_{1}(x,y)dx dy=0 \bigg\},
\end{align*} 
where $\bv=(v_{1},v_{2})$.  Lemma \ref{gradperpontoalpha} follows as stated. One can also easily check that if $\phi(x,1)= \phi(x,-1)=0$, then $\int_{\mathbb T} \int_{-1}^{1}u_{1}(x,y)dy dx=0$. Indeed, using \eqref{psiphibc},
\begin{align*}
&\int_{\mathbb T}\int_{-1}^{1}u_{1}(x,y)dy dx = \int_{\mathbb T}\int_{-1}^{1} \psi_{y}(x,y) dy dx = \int_{\mathbb T} (\psi(x,1)-\psi(x,-1)) dx \\&=\int_{\mathbb T} (\phi(x,1)-\phi(x,-1)) dx= \int_{\mathbb T} (0 - 0) dx = 0.
\end{align*}
The proof of Lemma \ref{lapuniqalpha} is modified as follows.
\begin{proof}
Note that $\phi$ satsfies,
\begin{align*} 
&-\Delta(1-\alpha^{2}\Delta) \phi=0, \label{circfixuniq1chalpha} \numberthis\\
&\phi(x,-1)=\phi(x,1)=0, \label{circfixuniq2chalpha} \numberthis\\
&\phi_{yy}(x,-1)=\phi_{yy}(x,1)=0. \label{circfixuniq3chalpha} \numberthis
\end{align*}

Multiply  \eqref{circfixuniq1chalpha} by $\phi$ and integrate over the domain to get
\begin{align*}
0&=\int_{\mathbb T} \int_{-1}^{-1} \phi (-\Delta(1-\alpha^{2}\Delta) \phi) d y dx =  \int_{\mathbb T} \int_{-1}^{-1} \nabla \phi \cdot (\nabla(1-\alpha^{2}\Delta) \phi) d y dx,
\end{align*}
where boundary terms vanish by \eqref{circfixuniq2chalpha}.
Since \eqref{int_by_parts_alph_Eul} also holds true in this case, we have that \[ \int_{D} \nabla \phi \cdot \Delta \nabla \phi  d \mathbf{x} =\int_{D} -\nabla^{\perp} \phi \cdot \Delta (-\nabla^{\perp} \phi)  d \mathbf{x} = \int_{D} \bv \cdot \Delta \bv d \mathbf{x} = -\int_{D} |\nabla \bv|^{2}  d \mathbf{x}, \] 
where $\bv \in X_{\alpha}$ is the unique solution to $\bv=-\nabla^{\perp}\phi$, via Lemma \ref{gradperpontoalpha}.
Thus, we see that, 
\begin{align*} 
&\int_{D} \nabla \phi \cdot (1-\alpha^{2}\Delta) \nabla \phi  d \mathbf{x}= \int_{D} \nabla \phi \cdot \nabla \phi  d \mathbf{x} + \alpha^{2} \int_{D} |\nabla \bv|^{2}  d \mathbf{x} \\&=\int_{D} \bv \cdot \bv  d \mathbf{x} + \alpha^{2} \int_{D} |\nabla \bv|^{2}  d \mathbf{x} = 0,
\end{align*} from which we conclude that $\bv=0$.
It follows by Lemma \ref{gradperpontoalpha} that $-\nabla^{\perp} \phi =0$ on $D$ and hence $\nabla \phi=0$. Then $\phi$ is a constant, and is equal to $0$ by \eqref{circfixuniq2chalpha}.
\end{proof}
Lemma \ref{curl11alpha} follows as stated and Arnold's second theorem also follows as stated. 
\begin{remark}
The stability results considered above implicitly assume that a solution exists for all times $t$. If not, then one has the stability estimate for all times for which the solution exists. This is sometimes referred to in the literature as \textit{conditional stability} (see, for example, \cite[page 7]{HMRW85}).
\end{remark}
\subsection{Examples}
\begin{example}
\textit{Plane parallel shear flows and inflection points.}

(1) Suppose we have a plane parallel shear flow on $\mathbb T \times [-L_{1},L_{2}]$ induced by the profile $\bv^{0}=(V(y),0)$, with $V'(-L_{1})=V'(L_{2})=0$, $\bu^{0}=(U(y),0)$ where $U=V-\alpha^{2}V''$. %
We assume that $U(y)$ has no inflection point on $[-L_{1},L_{2}]$, i.e., $U''(y) \neq 0$ for every $y \in [-L_{1},L_{2}]$. We compute $-F'(\omega^{0}(x,y))= V(y)/U''(y)$. Therefore, as long as $U''(y) \neq 0$, we can always move to a reference frame where $V$ has the same sign as $U''$, i.e., find a constant $c$ such that $V(y)+c$ has the same sign as $U''$. Thus, one can see that $-F'$ satisfies \eqref{K1K2} and one has stability of this steady state by Arnold's first stability theorem \ref{ArnoldIalphaeuler}. We thereby have a sufficient condition for stability for a shear flow $V$, where $U$ does not have any inflection point. Rayleigh criterion for $\alpha$-Euler, see Proposition \ref{prop_Rayl_crit_alpha} and example \ref{rex}, guarantees linear stability for flows such that $U$ has no inflection point. Arnold's stability theorem \ref{ArnoldIalphaeuler} guarantees nonlinear Lyapunov stability in the norm in \eqref{esteuleralpha} thus generalizing appropriately the Rayleigh criterion.

(2) Consider now plane parallel shear flows $V$ such that $U=V-\alpha^{2}V''$ has  inflection points but $V$ and $U''$ have the same sign everywhere.
We can also prove stability of a steady state $\bv^{0}=(V(y),0)$, with $V'(-L_{1})=V'(L_{2})=0$, $\bu^{0}=(U(y),0)$ where $U=V-\alpha^{2}V''$ such that $U''$ changes sign, but $V(y)/U''(y)$ has the same sign. Then, the ratio  $-F'(\omega^{0}(x,y))= V(y)/U''(y)$ is positive everywhere and one obtains stability of this steady state by Arnold's first stability Theorem \ref{ArnoldIalphaeuler}. Note that this generalizes the Fjortoft criterion, see Proposition \ref{prop_F_crit_alpha} and example \ref{fexample}, which guaranteed linear stability of these steady states. 
\end{example}
\begin{example}
\textit{Effect of regularization.}
This example illustrates the effect of regularization on the Arnold criterion. We present an example such that the Arnold stability Theorem \ref{ArnoldIalphaeuler} can be applied to conclude stability of the steady states for $\alpha$-Euler for every $\alpha > 0$ but the corresponding Arnold Theorem for Euler cannot be applied to conclude stability of the steady state for Euler equation obtained, formally, by putting $\alpha=0$ in the steady states for the $\alpha$-Euler. 
Suppose we have a plane parallel shear flow on $\mathbb T \times [-L_{1},L_{2}]$ induced by a profile $\bv^{0}=(V(y),0)$, with $V'(-L_{1})=V'(L_{2})=0$, $\bu^{0}=(U(y),0)$ where $U=V-\alpha^{2}V''$. 
Let $\phi^{0}$ be the stream function of velocity $\bv^{0}$, i.e., $V(y)=(\phi^{0})'(y)$. The boundary condition $V'(-L_{1})=V'(L_{2})=0$ implies that $(\phi^{0})''(L_{1})=(\phi^{0})''(L_{2})=0$. Also, assume that $\phi^{0}=(1+\alpha^{2})\omega^{0}$. Thus, $F'=(1+\alpha^{2})$.  Notice that
 $\omega^{0}=\curl (1-\alpha^{2}\Delta)\bv^{0}= - U'(y)$ Thus \[\omega^{0} = -\partial_{y}((1-\alpha^{2}\partial_{yy})V(y)) =  -\partial_{y}((1-\alpha^{2}\partial_{yy})\partial_{y}\phi^{0})=\alpha^{2}(\phi^{0})''''-(\phi^{0})''.\] 
Since $\omega^{0}=\frac{1}{1+\alpha^{2}}\phi^{0}$, we see that $\phi^{0}$ must satisfy the following differential equation,
\begin{equation}\label{phieqn}
\alpha^{2}(\phi^{0})''''-(\phi^{0})''-\frac{1}{1+\alpha^{2}}\phi^{0}=0, \quad (\phi^{0})''(L_{1})=(\phi^{0})''(L_{2})=0.
\end{equation}
Choose the difference $L_{2}-L_{1}$ in such a way that $-\Delta$ has minimum eigenvalue $1$ on the appropriate space $Y_{\alpha}$. Thus $-\Delta(1-\alpha^{2}\Delta)$ will have minimum eigenvalue $1+\alpha^{2}$. Since we have the inequality $\frac{1}{\lambda_{min,\alpha}}=\frac{1}{1+\alpha^{2}}< F'=1+\alpha^{2}$ for all $\alpha > 0$, by Arnold's second stability theorem \ref{Arnold II alpha}, one has stability of this steady state for all values of $\alpha>0$.
Notice that if we put $\alpha = 0$ and consider this as a steady state for the Euler equations, $\phi^{0}=\psi^{0}=$ and $\phi^{0}=(1+\alpha^{2})\omega^{0}$ becomes $\psi^{0}=\omega^{0}$. Thus $F'=1$ and since the minimum eigenvalue of $-\Delta$ in the appropriate subspace is $1$, the inequality $1/ \lambda_{min} < F'$ cannot be checked and stability of this steady state cannot be concluded by Arnold's second stabilty theorem for the Euler equations. In fact, stability holds in a restricted sense if perturbations are restricted to certain subspace, see \cite[p. 111]{MP94} for more details.
\end{example}
\begin{example}
\textit{Sinusoidal flows.}
One class of steady states for which the regularization seems to have no effect in terms of the Arnold criterion are the oscillating sinusoidal flows, i.e., steady states of the form $\phi^{0}(y)=\sin y$ and $\phi^{0}(y)=\sin my$ where $m > 1$ is an integer. The Arnold stability theorems cannot be used to conclude stability of these steady states for both Euler and $\alpha$-Euler. 
Consider the domain to be the two torus $\mathbb T^{2}$. For example, if $\phi^{0}(y)=\sin y$, then $\bv^{0}(y)=(\cos y,0)$, $\bu^{0}(y)=(1-\alpha^{2}\partial_{yy})\bv^{0}(y) = ((1+\alpha^{2})\cos y,0)$, $\omega^{0}(y)=-\partial_{y}(1+\alpha^{2})\cos y = (1+\alpha^{2})\sin y$. From this we can see that $F'=1/(1+\alpha^{2})$ and in order to check for stability we need $1/ \lambda_{min,\alpha} < 1/ (1+\alpha^{2})$ which does not hold because $\lambda_{min,\alpha}$ of the operator $-\Delta(1+\alpha^{2}\Delta)$ with domain $Y_{\alpha}$ as in Equation \eqref{yalphatorus} is equal to $1+\alpha^{2}$. One thus cannot conclude stability via Arnold's second Theorem \ref{Arnold II alpha} We note here the regularization does not have any effect whatsoever because for the Euler equation if $\psi^{0}(y)=\sin y$, $\bu^{0}=(\cos y,0)$, $\omega^{0}(y)=\sin y$. Thus $F'=1$ and $\lambda_{min}$ of the negative Laplacian $-\Delta$ acting on the appropriate subspace %
is also $1$ and thus one cannot check that $1/F' < \lambda_{min}$ which is required for stability. Thus the regularization doesn't seem to affect the ability of Arnold criterion to predict the stability of the steady state $\sin y$.
Similarly, if $\phi^{0}(y)=\sin my$, then $\bv^{0}(y)=(m \cos my,0)$. Then, $\bu^{0}(y)=(1-\alpha^{2}\partial_{yy})\bv^{0}(y)=(m(1+\alpha^{2}m^{2})\cos my,0)$. $\omega^{0}(y)=-\partial_{y}m(1+\alpha^{2}m^{2})\cos my = m^{2}(1+m^{2}\alpha^{2})\sin my$. From this we can see that $F'=1/ m^{2}(1+m^{2}\alpha^{2})$. One can check that the minimum eigenvalue of $-\Delta(1-\alpha^{2}\Delta)$ on the subspace $Y_{\alpha}$ described in Section \ref{twotorus} is given by $1+\alpha^{2}$. Thus, in order to check for stability we need $1/ \lambda_{min,\alpha}= 1/(1+\alpha^{2}) < F'= 1/ (m^{2}(1+m^{2}\alpha^{2}))$. This inequality cannot be checked for $m>1$ and thus one cannot conclude stability by Arnold's second stability Theorem \ref{Arnold II alpha}. We note here the regularization does not have any effect whatsoever because for in the case of the Euler equation if $\psi^{0}(y)=\sin my$, $\bu^{0}=(m\cos my,0)$, $\omega^{0}(y)=m^{2}\sin my$. Thus $F'=1/ m^{2}$ and $\lambda_{min}$ of the negative Laplacian $-\Delta$ acting on the appropriate subspace %
is also $1$. Thus we need to check if $1 < 1/m^{2}$ which cannot be true if  $m>1$, and thus, similar to the $\alpha$-Euler criterion, even for the Euler case, stability cannot be concluded via the Arnold's second stability theorem. Thus the regularization doesn't seem to affect the ability of Arnold criterion to predict the stability of the steady state $\sin m y$. This leads us to conjecture that these steady states are unstable even for the regularized $\alpha$-Euler equations. For more regarding stability of sinusoidal flows for the Euler equations, see \cite{BFY99}.
\end{example}


\end{document}